\documentclass[a4paper,11pt, reqno]{amsart}
\usepackage{amsmath,amsfonts,amsthm,amssymb,enumerate}
\usepackage{graphicx}

\numberwithin{equation}{section}

\newtheorem{theorem}{Theorem}[section]
\newtheorem{corollary}[theorem]{Corollary}
\newtheorem{proposition}[theorem]{Proposition}

\newtheorem{lemma}[theorem]{Lemma}

{
\theoremstyle{definition}
\newtheorem{definition}[theorem]{Definition}

\newtheorem{remark}[theorem]{Remark}
}

\title{Torus-covering knot groups and their irreducible metabelian $SU(2)$-representations}
\author{Inasa Nakamura}
\address{Faculty of Electrical, Information and Communication Engineering, \newline
Institute of Science and Engineering, \newline
Kanazawa University, \newline 
Kakumamachi, Kanazawa, 920-1192, Japan}
\email{inasa@se.kanazawa-u.ac.jp}

\address{{\rm (Current address)} Department of Mathematics, Information Science and  Engineering,  
Saga University,  
1 Honjomachi, Saga, 840-8502, Japan}
\email{inasa@cc.saga-u.ac.jp}


\begin{document}  
\begin{abstract}
A torus-covering $T^2$-knot is a surface-knot of genus one determined from a pair of commutative braids. For a torus-covering $T^2$-knot $F$, we determine the number of irreducible metabelian $SU(2)$-representations of the knot group of $F$ in terms of the knot determinant of $F$. 
It is similar to the result due to Lin for the knot group of a classical knot. Further, we investigate the number of irreducible metabelian $SU(2)$-representations using Fox's $p$-colorability. 
\end{abstract}
\maketitle

\section{Introduction}\label{sec1}

Let $\rho$ be a representation of a group $G$ on a vector space. We call $\rho$ a {\it metabelian} representation if $\rho([G, G])$ is abelian, where $[G,G]$ is the commutator subgroup of $G$. 
We call $\rho$ {\it irreducible} if it has no non-trivial invariant subspaces. 

A {\it (classical) link} is the image of a smooth embedding of several circles into the 3-sphere $S^3$. In particular, when it consists of one component, it is called a {\it (classical) knot}. 
The {\it knot group} of a knot is the fundamental group of the knot complement. 
Lin \cite{Lin} studied irreducible metabelian $SU(2)$-representations of the knot group $\pi_1(S^3 \backslash K)$ of a classical knot $K$, and showed the following result. 

\begin{theorem}[Lin \cite{Lin}]\label{thm-Lin}
Let $K$ be a classical knot in $S^3$. 
Then the number of conjugacy classes of irreducible metabelian $SU(2)$-representations of $\pi_1(S^3 \backslash K)$ is 
$\displaystyle \frac{|\Delta_K(-1)|-1}{2}$. 
Here $|\Delta_K(-1)|$ is the knot determinant of $K$, which is the absolute value of $\Delta_K(-1)$, where $\Delta_K(t)$ is 
the Alexander polynomial of $K$. 
\end{theorem}

There are other researches that 
determine such a number for metabelian representations 
in terms of the knot determinant. 
The same statement with Theorem \ref{thm-Lin} holds true for 
irreducible metabelian $SL(2, \mathbb{C})$-representations of the knot group of a classical knot, due to Nagasato \cite{Nagasato}. And 
Fukuda \cite{Fukuda} studied irreducible metabelian $SL(2, \mathbb{C})$-representations of the knot group of a certain surface-knot called a branched twist spin, which is determined from a classical knot and a pair of integers. 

In this paper, we investigate irreducible metabelian $SU(2)$-representations of  the knot group of a surface-knot called a \lq\lq torus-covering $T^2$-knot''. 
A {\it surface-link} is the image of a smooth embedding of a closed surface into the 4-sphere $S^4$. When it consists of one component, it is called a {\it surface-knot}. In this paper, we treat orientable surface-knots, and we assume that classical links/knots and surface-links/knots are oriented. Two classical links/knots or surface-links/knots are {\it equivalent} if they belong to the same isotopy class. 
Joseph \cite{Joseph} studied Alexander ideals of surface-knots, which may be nonprincipal, and gave the notion of the determinant of a surface-knot.  
For a surface-knot $F$ with the Alexander ideal $\Delta(F)$, the knot determinant of $F$, denoted by $\Delta(F)|_{t=-1}$, is defined as the nonnegative generator of the $\mathbb{Z}$-ideal $\{f(-1) \mid f(t) \in \Delta(F)\}$. 
In this paper, we treat torus-covering $T^2$-knots. 
An embedded torus is called a {\it $T^2$-knot}, and 
a torus-covering $T^2$-knot is a $T^2$-knot determined from a pair of commutative $n$-braids $(a,b)$ called basis braids, where $n$ is a positive integer, which will be denoted by $\mathcal{S}_n(a,b)$. For a braid $a$, the closure of $a$ is the link/knot obtained from $a$ by connecting each $i$th initial point and $i$th terminal point by a trivial arc.

\begin{theorem}\label{thm2-1}
Let $F=\mathcal{S}_n(a,b)$ be a torus-covering $T^2$-knot with basis $n$-braids $(a, b)$ such that the closure of the braid $a$ is a knot. 
Then the number of conjugacy classes of irreducible metabelian $SU(2)$-representations of $\pi_1(S^4 \backslash F)$ is 
$\displaystyle \frac{(\Delta(F)|_{t=-1})-1}{2}$. Here, $\Delta(F)|_{t=-1}$ is the knot determinant of $F$.  
\end{theorem}
Further we study our theme using Fox's $p$-coloring. 
Alexander module is closely related to Alexander quandle coloring of a classical knot diagram or a surface-knot diagram, and $|\Delta_K(-1)|$ or $\Delta(F)|_{t=-1}$ is closely related to Fox's $p$-coloring. 
We say that a classical knot or a surface-knot is {\it $p$-colorable} if its diagram admits a non-trivial $p$-coloring such that colors of arcs or sheets generate the dihedral quandle $R_p$ by the binary operation of $R_p$. 
For a braid $a$, we denote by $\hat{a}$ the closure of $a$. 

\begin{theorem}\label{thm2-2}
Let $\mathcal{S}_n(a,b)$ be a torus-covering $T^2$-knot with basis $n$-braids $(a, b)$ such that the closure of the braid $a$ is a knot. 
If $\mathcal{S}_n(a,b)$ is $|\Delta_{\hat{a}}(-1)|$-colorable, then the number of conjugacy classes of irreducible metabelian $SU(2)$-representations of $\pi_1(S^4 \backslash \mathcal{S}_n(a,b))$ is 
$\displaystyle \frac{|\Delta_{\hat{a}}(-1)|-1}{2}$. 
\end{theorem}

Let $p$ be an integer. 
For a surface-knot $F$, we denote by $| \mathrm{Col}_p(F)|$ the number of $p$-colorings of a diagram of $F$. 

\begin{theorem}\label{thm2-3}
Let $\mathcal{S}_n(a,b)$ be a torus-covering $T^2$-knot with basis $n$-braids $(a, b)$ such that the closure of the braid $a$ is a knot. Let $p$ be an odd prime.  
If $\mathcal{S}_n(a,b)$ is $p$-colorable, and not $q$-colorable for any integer $q \neq p$ $(q>1)$, 
then the number of conjugacy classes of irreducible metabelian $SU(2)$-representations of $\pi_1(S^4 \backslash \mathcal{S}_n(a,b))$ is 
$\displaystyle \frac{|\mathrm{Col}_p(\mathcal{S}_n(a,b))|-p}{2p}$. 
\end{theorem}

Together with results in \cite{N5}, Theorem \ref{thm2-3} implies the following corollary. 
Let $p$ be an odd prime, and let $n>1$ be an integer.  
Let $\sigma_1, \ldots, \sigma_{n-1}$ be the standard generators of the $n$-braid group. 
Let $c$ be an $n$-braid which is presented by $\sigma_{s(1)}^{\epsilon_1 p} \sigma_{s(2)}^{\epsilon_2 p} \ldots, \sigma_{s(n-1)}^{\epsilon_{(n-1)} p}$, where $s$ is a permutation of the set $\{1, \ldots, n-1\}$, and $\epsilon_i \in \{+1, -1\}$ $(i=1, \ldots, n-1)$. Remark that $\hat{c}$ is a knot. 
Put $l=2$ (respectively, $p$) if $n$ is odd (respectively, even). 
Let $\tau$ be a full-twist of $n$ parallel strings.  

\begin{corollary}\label{cor2-4}
For any integer $m$, the number of conjugacy classes of irreducible metabelian $SU(2)$-representations of $\pi_1(S^4\backslash \mathcal{S}_n(c,\tau^{lm}))$ is 
$\displaystyle \frac{p^{n-1}-1}{2}$. 
\end{corollary}
 
The paper is organized as follows. In Section \ref{sec2-8}, we review the Alexander ideal and the knot determinant of a classical knot and a surface-knot. In Section \ref{sec4}, we review torus-covering $T^2$-knots. In Section \ref{sec5}, We review Lin's presentation of the knot group of a classical knot $K$, using a Seifert surface of $K$. In Section \ref{sec3-pres}, we review a presentation of the knot group of a classical knot, using Artin's automorphism. 
In Section \ref{section-0602}, we give a presentation of the knot group of a torus-covering $T^2$-knot, and in Section \ref{sec-proof}, we show Theorem \ref{thm2-1}. 
In Section \ref{sec7}, we review quandle colorings and we show Theorems \ref{thm2-2} and \ref{thm2-3} and Corollary \ref{cor2-4}.

\section{The Alexander ideal and the knot determinant}\label{sec2-8}
In this section, we review the Alexander ideal and the knot determinant of a classical knot and a surface-knot \cite{Alexander, CF, Fox, Joseph, Kamada02, Kawauchi, Lickorish}. 

\subsection{Presentation matrices}\label{section2-1a}
Let $\Lambda=\mathbb{Z}[t, t^{-1}]$, the Laurent polynomial ring. 
Let $M$ be a finitely generated $\Lambda$-module. Then, since $\Lambda$ is Noetherian, we have an exact sequence $\Lambda^n \to \Lambda^m \to M \to 0$ for some positive integers $n,m$. A {\it presentation matrix} of $M$ is the $(n,m)$-matrix representing the homomorphism $\Lambda^n \to \Lambda^m$. 
Though presentation matrices of $M$ are not unique, they are related by a finite sequence of the following operations. Let $P$ be a presentation matrix. 

\begin{enumerate}
\item
Permutation of rows or columns.
\item
Multiplication of a row or a column by a unit of $\Lambda$.
\item
Addition to some row a $\Lambda$-linear combination of other rows.
\item
Addition to some column a $\Lambda$-linear combination of other columns.
\item
Exchange of $P$ with $\begin{pmatrix} P \\ * \end{pmatrix}$, where $*$ is a $\Lambda$-linear combination of rows of $P$, and the inverse operation. 
\item
Exchange of $P$ with $\begin{pmatrix} P & 0 \\ * & 1 \end{pmatrix}$, where $*$ is an arbitrary row, and the inverse operation. 
\end{enumerate}
We say that presentation matrices are {\it equivalent} if they are related by these operations.

\subsection{Fox calculus and the Alexander matrix}\label{section3-2} 
We review Fox's free differential calculus. 
Let $F_m$ be a free group with $m$ generators $x_1, \ldots, x_m$.
For each $j \in \{1, \ldots, m\}$, there is a unique map
\[
\frac{\partial}{\partial x_j}: F_m \to \mathbb{Z} F_m
\]
determined from the following conditions:

\begin{enumerate}
\item $\displaystyle \frac{\partial x_i}{\partial x_j}=\delta_{ij}$ (Kronecker delta), \\
\item $\displaystyle \frac{\partial (uv)}{\partial x_j}=\frac{\partial u}{\partial x_j}+u \frac{\partial v}{\partial x_j}$.
\end{enumerate}
We denote by the same notation $\partial/ \partial x_j$ the induced ring homomorphism $\mathbb{Z} F_m \to \mathbb{Z} F_m$, which is called 
 the {\it free derivative} with respect to $x_j$. 
 
Let $G$ be a group with a presentation with $m$ generators $x_1, \ldots,x_m$ and $n$ relations $r_1, \ldots, r_n$ such that $G/[G,G] $ is an infinite cyclic group with a generator $t$, where $[G, G]$ is the commutator subgroup. 
 Let $\phi: F_m \to G$ be the projection, and let $\psi: G \to G/[G,G]$ be the abelianization map. 
 We denote by the same notation $\phi$ and $\psi$ the induced ring homomorphisms $\mathbb{Z} F_m \to \mathbb{Z} G$ and $\mathbb{Z} G \to \mathbb{Z} [t, t^{-1}]=\Lambda$, respectively. 
The {\it Alexander matrix} $A(t)$ of $G$ is given by 
\[
A(t)=\left(\psi \phi \left( \frac{\partial r_i}{\partial x_j} \right) \right)_{1 \leq i \leq n, 1 \leq j \leq m}.
\]

\subsection{The Alexander ideal/polynomial and the determinant of a classical knot}\label{sec3-3}
Let $K$ be a classical knot in $S^3$. 
Among the two candidates of a generator $t$ of $H_1(S^3 \backslash K; \mathbb{Z}) \cong \mathbb{Z}$, we take as $t$ the one that is represented by a circle $c$ which encircles an arc of $K$ in a positive direction; that is, $c$ bounds a small disk $D$ such that $D$ intersects with $K$ transversely at one point $x$ and $(\mathbf{v}_1, \mathbf{v}_2, \mathbf{v}_3)$ coincides with the positive orientation of $S^3$, where $(\mathbf{v}_1, \mathbf{v}_2)$ are tangent vectors of $D$ at $x$ representing the orientation of $D$ induced from $c$, and $\mathbf{v}_3$ is a tangent vector of $K$ at $x$; and here the positive orientation of $S^3$ is given by the standard basis of the tangent space of $S^3$ at $x$. 

We denote the knot complement $S^3 \backslash K$ by $E$. Let $\mathbf{p}: \widetilde{E} \to E$ be the universal abelian covering, that is, the covering corresponding to the kernel of the abelianization map $\psi: \pi_1(E, \star) \to H_1(E; \mathbb{Z})$, where $\star$ is a base point. 
 The first homology group $H_1(E; \mathbb{Z}) =\langle t \rangle$ acts on $\widetilde{E}$ as the covering transformation group. 
We identify $\mathbb{Z}H_1(E;\mathbb{Z})$ with $\Lambda=\mathbb{Z}[t,t^{-1}]$, and we regard $H_1(\widetilde{E}, \mathbf{p}^{-1}(\star); \mathbb{Z})$ as a $\Lambda$-module, which is called the {\it Alexander module}. 

The {\it Alexander ideal} $\Delta(K)$ is the first elementary ideal of the Alexander module $M$, that is the $\Lambda$-ideal generated by $(m-1)$-minors of a presentation matrix $P$ of $M$, where $m$ is the number of columns of $P$. We define  
$\Delta(K)=\Lambda$ if $m=1$ and $\Delta(K)=0$ if $m-1>n$, where $n$ is the number of rows of $P$. 
The Alexander ideal is principal, and the generator is called 
the {\it Alexander polynomial}.  The Alexander polynomial is unique up to multiplication by a unit of $\Lambda$. 
A presentation matrix of the Alexander module $M$ is obtained as an  Alexander matrix, as follows.  Let $A(t)$ be the Alexander matrix of a presentation $\langle x_1, \ldots, x_m \mid r_1, \ldots, r_n \rangle$ of $\pi_1(S^3\backslash K)$, and 
let $f: \Lambda^n \to \Lambda^m$ be the linear map presented by $A(t)$.  
Then, $M$ is isomorphic to $\mathrm{Coker} f$; and  
$A(t)$ is a presentation matrix of the Alexander module.  

Let $\Delta_K(t)$ be the Alexander polynomial of a knot $K$. The {\it knot determinant} is the absolute value of the number obtained from putting $t=-1$ in $\Delta_K(t)$, that is, $|\Delta_K(-1)|$. The knot determinant $|\Delta_K(-1)|$ equals $\Delta(K)|_{t=-1}$; here, for an ideal $I \subset \Lambda=\mathbb{Z}[t, t^{-1}]$, $I|_{t=-1}$ denotes  the nonnegative generator of the ideal $\{f(-1) \mid f(t) \in I \} \subset \mathbb{Z}$ (see Remark \ref{remark-0216}). 

\subsection{The Alexander ideal and the determinant of a surface-knot}\label{section3-4}
For a surface-knot $F$, the Alexander module and the Alexander ideal $\Delta(F)$ are defined by the same way as in Section \ref{sec3-3}, where we denote by $E$ the knot complement $S^4 \backslash F$. The {\it knot determinant} is defined as $\Delta(F)|_{t=-1}$ \cite{Joseph}. 

\begin{remark}\label{remark-0216}
We remark that while Alexander ideals of classical knots are principal, surface-knots may have nonprincipal Alexander ideals. 
And an obstruction to 0-concordance of surface-knots are given in terms of nonprincipal Alexander ideals; see \cite{Joseph}.
\end{remark}

\section{Torus-covering $T^2$-knots}\label{sec4}
In this section, we review a torus-covering $T^2$-knot \cite{N}. 
A torus-covering $T^2$-knot is an embedded torus in the form of an unbranched covering over a standard torus, and it is determined from a pair of commutative braids, called basis braids.

Let $T$ be a torus standardly embedded in $S^4$, that is, $T$ is the boundary of an unknotted solid torus in $S^3 \times \{0\}=S^4 \cap (\mathbb{R}^4 \times \{0\})$ for the unit 4-sphere $S^4 \subset \mathbb{R}^5$. 
Let $N(T)$ be a tubular neighborhood of $T$ in $S^4$. 
\begin{definition}
A surface-link $F$ in $S^4$ is called a {\it torus-covering $T^2$-link} if it is contained in $N(T) \subset S^4$ and 
$\mathbf{p} |_F: F \to T$ is an unbranched covering map, where $\mathbf{p}: N(T) \to T$ is the natural projection. 
In particular, when it is a surface-knot, it is called a {\it torus-covering $T^2$-knot}. 
\end{definition}

By definition, a torus-covering $T^2$-knot is a $T^2$-knot: the image of an embedding of a torus. 
Let $F$ be a torus-covering $T^2$-knot. We identify $N(T)=D^2 \times S^1 \times S^1 \subset D^3 \times S^1 \subset S^4$, where the second $S^1$ in $D^2 \times S^1 \times S^1$ is $S^1$ in $D^3 \times S^1$.  Let $*$ be a base point of $S^1$, 
and let $\star=(\star',*) \in D^2 \times S^1 \subset D^3$ for a point $\star' \in D^2$, and let $\diamond=(\star,*)$ be a base point of $D^3 \times \{*\} \subset S^4$. We assume that the base point $\diamond$ is in the complement of $F$. Let $\mathbf{m}=S^1 \times \{*\}$ and $\mathbf{l}= \{*\} \times S^1$, a meridian and a longitude of $T$ with the base point $*'=(*,*)$.  The condition that $F$ is an unbranched cover over $T$ implies that the intersections $F \cap \mathbf{p}^{-1}(\mathbf{m})$ and $F \cap \mathbf{p}^{-1}(\mathbf{l})$ are closures of classical braids in solid tori $\mathbf{p}^{-1}(\mathbf{m})=D^2 \times S^1 \times \{*\}$ and $\mathbf{p}^{-1}(\mathbf{l})=D^2 \times \{*\} \times S^1$. 
Regarding that the solid tori are pasted at the 2-disk $\mathbf{p}^{-1}(*')=D^2 \times \{*'\}$, we have a pair of braids, called {\it basis braids}. 

Let $n$ be a positive integer.
For $n$-braids $a$ and $b$, we say $a$ and $b$ {\it commute} or $a$ and $b$ are {\it commutative} if $ab=ba$ as elements of the $n$-braid group. 
For a torus-covering $T^2$-link, basis braids commute, and for any 
pair of commutative braids $(a,b)$, there exists a unique torus-covering $T^2$-link with basis braids $(a,b)$. For a pair of commutative $n$-braids $(a,b)$, we denote by $\mathcal{S}_n(a,b)$ the torus-covering $T^2$-link with basis $n$-braids $(a,b)$.

 \section{Lin's presentation of a classical knot group}\label{sec5}
 In this section, we review Lin's presentation of a classical knot group (Lemma \ref{lem-Lin})  \cite{Lin}. 

\subsection{The Seifert surface associated with the closure of a braid}\label{sec-s}
 Let $K$ be a classical knot in $S^3$. 
 A {\it Seifert surface} of $K$ is a compact, connected, oriented surface $S$ embedded in $S^3$ whose boundary is $K$. Though we can discuss the argument in this section for a \lq\lq free'' Seifert surface, 
since it suffices to  consider closures of braids to show our results, we explicitly take the Seifert surface associated with the closure of a braid, in such a way as follows. For a braid $a$, we denote its closure by $\hat{a}$.  

Let $a$ be an $n$-braid for a positive integer $n$, with the starting point set $Q_n$. 
We consider the closure of the trivial $n$-braid $e$ with the same starting point set $Q_n$, and  
take $n$ disjoint disks in $S^3$ that bound the closure of $e$. 
Then, to these disks, we attach half-twisted bands each of which corresponds to a crossing of the braid, presented by $\sigma_i$ or $\sigma_i^{-1}$ $(i\in \{1,\ldots, n-1\})$. The resulting surface consisting of the disks and bands is a Seifert surface, which will be called the {\it Seifert surface associated with $\hat{a}$}. See Fig. \ref{fig6-1}. 

\begin{figure}[ht]
\includegraphics[height=3cm]{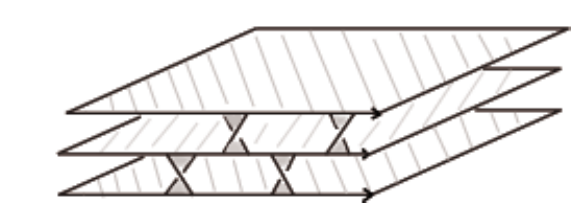}
\caption{The Seifert surface associated with the closure of a braid.}\label{fig6-1}
\end{figure}

\subsection{Lin's presentation}\label{sec-l}
Let $K$ be a classical knot. 
For our convenience, we consider $K$ in the form of the closure of a braid.  We remark that by Alexander theorem, any knot is presented by the closure of a braid. 

Let $S$ be the Seifert surface associated with $K$. 
A {\it spine} of $S$ is a bouquet of circles in $S$ such that it is a deformation retract of $S$, and half the number of circles is called the {\it genus} of $S$. 
Let $W$ be a spine of $S$, and let $g$ be the genus of $S$. Denote by $a_1, b_1, \ldots, a_g, b_g$ the oriented circles in $W$. We assume that $W$ is in the interior of $S$, and we denote by $*$ the base point of $W$. 
Let $S\times [-1, 1]$ be a bicollar of $S$ such that $S$ is identified with $S \times \{0\}$, and $[-1, 1]$ has an orientation coherent with the normal direction of $S$. Let $a_i^{\pm}=a_i \times \{\pm 1\}$ and $b_i^{\pm}=b_i \times \{\pm 1\}$ ($i=1,\ldots, g$). Let $c$ be the arc $\{*\} \times [-1,1]$, and let $c'$ be a trivial arc in $S^3 \backslash S \times (-1,1)$ connecting $\{*\} \times \{1\}$ and $\{*\} \times \{-1\}$. 
We consider $S^3 \backslash (S \times [-0.5,0.5])$. Since $W$ bounds $2g$ disks joined at the base point $*$,  the closure of $S^3 \backslash (S \times [-0.5,0.5])$ is a handlebody. Hence, the fundamental group $\pi_1 (S^3 \backslash (S \times [-0.5,0.5])$ is a free group. Choose a basis $x_1, \ldots, x_{2g}$ of $\pi_1 (S^3 \backslash (S \times [-0.5,0.5]))$ with a base point $\star=\{*\} \times \{1\}$. Let $\alpha_1, \ldots, \alpha_{2g}$ (respectively, $\beta_1, \ldots, \beta_{2g}$) be words in $x_1, \ldots, x_{2g}$ that present homotopy classes represented by the oriented circles $a_1^+, \ldots, a_g^+, b_1^+, \ldots, b_g^+$ (respectively, the conjugate by $c'$ of $a_1^-, \ldots, a_g^-, b_1^-, \ldots, b_g^-$) in $\pi_1 (S^3 \backslash S \times [-0.5,0.5], \star)$. We call $\alpha_j$ and $\beta_j$ $(j=1, \ldots, 2g)$ the {\it induced words}. 
A {\it meridian} is an oriented circle obtained from $c$ by connecting the end points by $c'$. 
Then, by van Kampen theorem, we have the following. 

\begin{lemma}[\cite{Lin}]\label{lem-Lin}
 Let $K$ be a knot in the form of the closure of a braid $a$. 
 Then, the knot group  $\pi_1(S^3\backslash K, \star)$ has a presentation 
\[
\langle \,x_1, x_2, \ldots, x_{2g}, z \mid z\alpha_j z^{-1}=\beta_j \quad (j=1,2,\ldots, 2g)\,\rangle, 
\]
where $g$ is the genus of the Seifert surface $S$ associated with $K=\hat{a}$, and $x_1, \ldots, x_{2g}$ are a basis of $\pi_1(S^3\backslash S \times [-0.5,0.5], \star)$, and $\alpha_j$ and $\beta_j$ are the {\it induced words} associated with $S$ and its spine, and $z$ is represented by a meridian of $K$. 
\end{lemma}

\section{Presentation of a classical knot group in terms of Artin's automorphism.}\label{sec3-pres}

In this section, we review a presentation of the group of the closure of a braid, in terms of Artin's automorphism \cite{CF, Kawauchi, Lickorish}. 
Let $n$ be a positive integer. Let $B_n$ be the $n$-braid group, and let $\sigma_1, \ldots, \sigma_{n-1}$ be the standard generators of $B_n$. 
For an $n$-braid $a$ in a cylinder $D^2 \times [0,1]$ with the starting point set $Q_n$, 
let $\{ h_u^a\}_{u \in [0,1]}$ be an isotopy of $D^2$ rel $\partial D^2$ such that $\cup_{u \in [0,1]} h_u(Q_n) \times \{u\}=a$. 
Then, the homeomorphism $h_1^a$ permutes the $n$-punctures $Q_n$ in $D^2$, carrying the starting point set to the terminal point set; so we regard it as $h_1^a: (D^2, Q_n) \to (D^2, Q_n)$. Let $\star' \in \partial D^2$ be a base point of $D^2 \backslash Q_n$. The induced map $h_{1*}^{a}$, which will be denoted by $\mathcal{A}^a : \pi_1(D^2 \backslash Q_n, \star') \to \pi_1(D^2 \backslash Q_n, \star')$ is an automorphism of $F_n=\pi_1(D^2 \backslash Q_n, \star')$, called {\it Artin's automorphism}; note that $F_n$ is a free group of rank $n$, generated by the standard generators of $\pi_1(D^2 \backslash Q_n, \star')$. Here, the standard generators $t_1, \ldots, t_n$ of $\pi_1(D^2 \backslash Q_n, \star')$ are given in such a way as follows: We identify $D^2$ with $[0,n+1] \times [-1,1]$, and we identify $Q_n$ with a set of points $\{ q_1, \ldots, q_n\}$ such that $q_i$ is given by $(i,0)$ $(i=1, \ldots, n)$, and we identify the base point $\star'$ with the point $(1, 1)$; then, the $i$th generator $t_i$ is represented by a closed path which starts from the base point, goes straight to the $i$th point $q_i$, encircles it once anti-clockwise, and returns straight to the base point $(i=1, \ldots, n)$. 

Artin's automorphisms give a homomorphism $B_n \to \mathrm{Aut}(F_n)$, which is faithful.  The generator $\sigma_i$ and its inverse $(i=1,\ldots, n-1)$ act on
$F_n=\langle t_1, \ldots, t_n \rangle$ as follow:  
\begin{eqnarray*}
\mathcal{A}^{\sigma_i}(t_j)=\begin{cases}
t_i t_{i+1} t_i^{-1} & \text{ if $j=i$}\\
t_i & \text{ if $j=i+1$}\\
t_j & \text{ otherwise}, 

\end{cases}
\end{eqnarray*}

and 

\begin{eqnarray*}
\mathcal{A}^{\sigma_i^{-1}}(t_j)=\begin{cases}
t_{i+1}  & \text{ if $j=i$}\\
t_{i+1}^{-1}t_{i}t_{i+1} & \text{ if $j=i+1$}\\
t_j & \text{ otherwise}. 

\end{cases}
\end{eqnarray*}

We take the closure of $a$ in such a way that 
$\hat{a} \subset D^2 \times S^1 \subset S^3$, with $S^1=[0,1]/0 \sim 1$. Let $*=0$ be a base point of $S^1$, and 
put $\star=(\star', *)$, which will be taken as a base point of $S^3 \backslash \hat{a}$. 
By van Kampen theorem, the knot/link group of $\hat{a}$ has the following presentation. 
  
  \begin{lemma}\label{lem-braid}
For an $n$-braid $a$, the knot/link group $\pi_1(S^3\backslash \hat{a}, \star)$ has a presentation 
\[
\langle \,t_1, t_2, \ldots, t_{n} \mid t_i=\mathcal{A}^a (t_i) \quad (i=1,2,\ldots, n)\,\rangle,
\]
where $t_1, \ldots, t_n$ are the standard generators of the free group $\pi_1(D^2 \backslash Q_n \times \{*\}, \star)$, and $\mathcal{A}^a$ is Artin's automorphism associated with $a$. 
  \end{lemma}

\section{Presentation of the knot group of $\mathcal{S}_n(a,b)$}
\label{section-0602}
 
In this section, we give a presentation of the knot group of $\mathcal{S}_n(a,b)$ such that $\hat{a}$ is a knot. We consider $\mathcal{S}_n(a,b)$ with the construction described in Section \ref{sec4}.  

We consider two presentations $G$ and $G'$ of the group $\pi_1 (D^3 \times \{*\} \backslash \hat{a})$, given by Lemmas \ref{lem-Lin} and \ref{lem-braid}, respectively. 
We arrange that the base point $\star \in D^3$ with respect to $G$ coincides with that with respect to $G'$. Then, we have presentations of $\pi_1 (D^3 \times \{*\} \backslash \hat{a})$ with a base point $\diamond=(\star, *)$, 
as follow: 
\[
G=\langle \,x_1, x_2, \ldots, x_{2g}, z \mid z\alpha_j z^{-1}=\beta_j \quad (j=1,2,\ldots, 2g)\,\rangle\]
of Lemma \ref{lem-Lin}, and 
\[G'=\langle \,t_1, t_2, \ldots, t_n \mid t_i=\mathcal{A}^a (t_i) \quad (i=1,2,\ldots, n)\,\rangle\]
of Lemma \ref{lem-braid}. 
Let $F(\cdots)$ denote the free group generated by the corresponding letters.
Let 
\begin{eqnarray*}
&&\phi: F(x_1, \ldots, x_{2g},z) \to G,\label{eq-phi}\\
&&\phi': F(t_1, \ldots, t_n) \to G'
\end{eqnarray*}
 be projections. 
Let $u_j \in F(t_1, \ldots, t_n)$ $(j=1,\ldots, 2g)$ be a word satisfying $\phi'(u_j)=\phi(x_j)$. Then, take a word $A^b(x_j) \in F(x_1, \ldots, x_{2g}, z)$ satisfying $\phi(A^b(x_j))=\phi'(\mathcal{A}^b(u_j))$. 
Similarly, considering $z$ instead of $x_j$, we take $A^b(z) \in F(x_1, \ldots, x_{2g},z)$; for $v\in F(t_1, \ldots, t_n)$ satisfying $\phi'(v)=\phi(z)$, $\phi(A^b(z))=\phi'(\mathcal{A}^b(v))$. 
We call $A^b(x_j)$ and $A^b(z)$ {\it words associated with Artin's automorphism $\mathcal{A}^b$}. 

\begin{lemma}\label{lem-k}
Let $F=\mathcal{S}_n(a,b)$ be a torus-covering $T^2$-knot with basis $n$-braids $(a, b)$ such that the closure of the braid $a$ is a knot. 
Then, the knot group $\pi_1(S^4 \backslash \mathcal{S}_n(a,b), \diamond)$ has a presentation 
\begin{equation}\label{eq-*}
\left\langle 
\begin{matrix}  x_1, x_2, \ldots, x_{2g}, z \end{matrix}
 \left| 
\begin{matrix}
 \ z\alpha_j z^{-1}=\beta_j,&\\
 \ x_j=A^b(x_j),& \\
 z=A^b(z) &\,(j=1,2,\ldots, 2g)
\end{matrix} \right. 
\right\rangle,
\end{equation}
where $g$ is the genus of the associated Seifert surface $S$ of $\hat{a} \subset D^3 \times \{ * \}$, and $x_1, \ldots, x_{2g}$ are a basis of $\pi_1(D^3 \times \{  *\} \backslash S, \diamond)=F(x_1, \ldots, x_{2g})$, and 
$\alpha_j$ and $\beta_j$ are the induced words associated with $S$ and its spine, and $z$ is represented by a meridian of $\hat{a}$, and 
$A^b(x_j)$ and $A^b(z)$ are words associated with Artin's automorphism $\mathcal{A}^b$. 
\end{lemma}

\begin{proof}
We have 
\begin{equation}\label{eq-5-2}
\begin{split}
& \pi_1(D^3 \times \{*\} \backslash \hat{a}, \diamond) \\
& \quad = 
\langle \,x_1, x_2, \ldots, x_{2g}, z \mid z\alpha_j z^{-1}=\beta_j \quad (j=1,2,\ldots, 2g)\,\rangle, 
\end{split}
\end{equation}
where $x_1, \ldots, x_{2g}$ are generators of $\pi_1(D^3 \times \{  *\} \backslash S, \diamond)=F(x_1, \ldots, x_{2g})$ for the associated Seifert surface $S$ of $\hat{a} \subset D^3 \times \{ * \}$, and 
$\alpha_j$ and $\beta_j$ are the induced words associated with $S$  and its spine, and $z$ is represented by a meridian. 
Let $S^1 =[0,1]/0 \sim 1$ with the base point $*=0$. 
We take isotopies of $D^2$ rel $\partial D^2$, $\{h_u^a\}_{u \in [0,1]}$ and $\{h_u^b\}_{u \in [0,1]}$, that are associated with $a$ and $b$, respectively, such that $h_1^a$ and $h_1^b$ commute. 
Then, the space $(D^3 \times S^1) \backslash \mathcal{S}_n(a,b)$ is a mapping torus written as 
\begin{equation}
(D^3 \times S^1) \backslash \mathcal{S}_n(a,b) \cong 
((D^3 \times \{*\}) \backslash \hat{a}) \times [0,1]/(x,0) \sim (f(x), 1),
\end{equation}
where 
\begin{equation}
f(x)=\begin{cases} (h_u^{a}\circ h_1^b \circ (h_u^a)^{-1} (x'), u, *) & \text{if $x=(x',u, *) \in (D^2  \times S^1 \times \{*\})  \backslash \hat{a}$} \\
x & \text{otherwise.} \end{cases}
\end{equation}
The monodromy is given by $f$ which maps a closed path representing $x_j \in \pi_1(D^3 \times \{ *\} \backslash S, \diamond)=F(x_1, \ldots, x_{2g})$ to the one whose homotopy class in $\pi_1(D^3 \times \{  *\} \backslash \hat{a}, \diamond)$ is represented by $A^b(x_j)$ $(j=1, \ldots, 2g)$, and $f$ maps a meridian to a closed path whose homotopy class is represented by $A^b(z)$. Hence, by van Kampen theorem, $\pi_1(D^3 \times S^1 \backslash \mathcal{S}_n(a,b), \diamond)$ 
has the presentation obtained from $\pi_1(D^3 \times \{ *\} \backslash \hat{a}, \diamond)*\mathbb{Z}$, where $\pi_1(D^3 \times \{ *\} \backslash \hat{a}, \diamond)$ has the presentation (\ref{eq-5-2}), by adding relations $sx_j s^{-1}=A^b(x_j)$ and 
 $s zs^{-1}=A^b(z)$: 
\begin{eqnarray*}
&&\pi_1(D^3 \times S^1 \backslash \mathcal{S}_n(a,b), \diamond) \\
&& \quad =
 \left\langle 
\begin{matrix}  x_1, x_2, \ldots, x_{2g}, z,s \end{matrix}
 \left| 
\begin{matrix}
 \ z\alpha_j z^{-1}=\beta_j,&\\
 \ sx_j s^{-1}=A^b(x_j),& \\
 s zs^{-1}=A^b(z) &\,(j=1,2,\ldots, 2g)
\end{matrix} \right. 
\right\rangle,
\end{eqnarray*}
where $s$ is represented by $\{\star'\} \times \mathbf{l}=\{\star\} \times S^1$, which will be denoted by $\mathbf{l}$. 

We obtain $\pi_1(S^4 \backslash \mathcal{S}_n(a,b), \diamond)$, as follows.
Let $E= \mathrm{Cl}(S^4 \backslash D^3 \times S^1)$. 
Then, $S^4 \backslash \mathcal{S}_n(a,b)=(D^3 \times S^1) \backslash \mathcal{S}_n(a,b) \cup _{\partial (D^3 \times S^1)} E$. 
The space $E$ is homotopy equivalent to $S^4 \backslash \mathbf{l}$. Since any two circles disjointly embedded in $S^4$ are unlinked, we see that $\pi_1(E)=0$. 
Further, $\pi_1(\partial E)=\pi_1(\partial (D^3 \times S^1)) \cong \mathbb{Z}$ such that the generator is represented by $\mathbf{l}$. 
Hence, by van Kampen theorem, 

\begin{eqnarray*}
&&\pi_1(S^4 \backslash \mathcal{S}_n(a,b), \diamond) \\
&&\quad =
 \left\langle 
\begin{matrix}  x_1, x_2, \ldots, x_{2g}, z,s \end{matrix}
 \left| 
\begin{matrix}
 \ z\alpha_j z^{-1}=\beta_j,&\\
 \ sx_j s^{-1}=A^b(x_j),& \\
 s zs^{-1}=A^b(z) & \\
s=1 & \,(j=1,2,\ldots, 2g)
\end{matrix} \right. 
\right\rangle,
\end{eqnarray*}
which is the required formula. 
\end{proof}

We denote relations of the presentation (\ref{eq-*}) by $r_i$ $(r=1,\ldots, 4g+1)$: 
\begin{eqnarray*}
&&r_1=z \alpha_1 z^{-1} \beta_1^{-1}, \label{rel}\\
&& \ldots \nonumber\\
&& r_{2g}=z \alpha_{2g} z^{-1} \beta_{2g}^{-1}, \nonumber\\
&& r_{2g+1}=x_1 (A^b(x_1) )^{-1},\nonumber\\
&& \ldots\nonumber\\
&& r_{4g}=x_{2g}  (A^b(x_{2g}) )^{-1}, \nonumber\\
&& r_{4g+1}=z  (A^b(z) )^{-1} \label{rel4g+1} \nonumber.
\end{eqnarray*}
For each $i \in \{1, \ldots, 4g+1\}$, 
we write 
\begin{equation}\label{rel2}
r_i=z^{\nu_1}w_{i1} z^{\nu_2} w_{i2} \cdots z^{\nu_{n_i}} w_{i n_i},
\end{equation}
where $w_{ik}$ $(k=1,\ldots, n_i)$ is a word consisting of $x_1, \ldots, x_{2g}$ such that $w_{ik} \neq 1$ for $k=1,\ldots, n_i-1$, and $\nu_k$ $(k=1,\ldots, n_i)$ is an integer such that $\nu_k \neq 0$ for $k=2,\ldots, n_i$. Let $\mu_{ik}^j$ be the sum of indices of $x_j$ in $w_{ik}$ ($i=1,\ldots, 4g+1$, $j=1, \ldots, 2g$, $k=1, \ldots, n_i$). 

For these relations (\ref{rel2}), 
we have the following lemmas. 
Let $\phi: F(x_1, \ldots, x_{2g}, z) \to \pi_1(S^4 \backslash \mathcal{S}_n(a,b))$ be the projection, 
and let $\psi: \pi_1(S^4 \backslash \mathcal{S}_n(a,b)) \to \langle t\rangle\cong \mathbb{Z}$ be the abelianization map. 
We denote by the same notation $\phi$ and $\psi$ the induced ring homomorphisms of group rings over $\mathbb{Z}$.

\begin{lemma}\label{lem3}
For $i=1,\ldots, 4g+1$,  
the sum of indices of $z$ equals zero: 
\begin{equation}\label{eq621}
\nu_1 + \nu_2+ \cdots +\nu_{n_i}=0.
\end{equation}
\end{lemma}

\begin{lemma} \label{lem4}
\begin{enumerate}[$(1)$]
\item
 For $i=1, \ldots, 4g+1$, $j=1, \ldots, 2g$, 
\begin{equation*}
\begin{split}
& \psi \phi\left(\frac{\partial r_i}{\partial x_j}  \right) \Big|_{t=-1}\\
& \quad = (-1)^{\nu_1} \mu_{i1}^j+ (-1)^{\nu_1+\nu_2} \mu_{i2}^j+ \cdots + (-1)^{\nu_1+\nu_2+\cdots +\nu_{n_i}} \mu_{i n_i}^j. 
\end{split}
\end{equation*}

\item
 For $i=1, \ldots, 4g+1$, 
\begin{equation*}
\psi\phi \left(\frac{\partial r_i}{\partial z}  \right)=0. 
\end{equation*}
\end{enumerate}
\end{lemma}

\begin{proof}[Proof of Lemma \ref{lem3}]
For $i=1,\ldots, 2g$, obviously $r_i$ satisfies the property $(\ref{eq621})$. We consider the case $i=2g+1, \ldots, 4g+1$. 
We consider the construction of $\mathcal{S}_n(a,b)$ given in Section \ref{sec4}. 
The first homology group $H_1(D^3 \times \{*\} \backslash \hat{a})$ is generated by the homotopy class represented by a meridian of $\hat{a}$, and a meridian encircles an arc of $\hat{a}$ in a positive direction, so the generator $t$ equals $\psi \phi(z)$ for the letter $z$. 
Since $x_j$ $(j=1,\ldots, 2g)$ is represented by a closed path in $D^3 \times \{*\} \backslash S \times [-0.5, 0.5]$, where $S$ is the Seifert surface associated with $\hat{a}$, we see that $\psi \phi(x_j)=1 \in \langle t\rangle$. 
Further, for $t_1, \ldots, t_n \in \pi_1( D^2\backslash Q_n)=F(t_1, \ldots, t_{n})$, $\psi\phi'(t_l)=t$ and $\psi \phi'(\mathcal{A}^{\sigma_m})(t_l)=t$ ($l=1, \ldots,n$, $m=1, \ldots, n-1$). 
Hence, 
\begin{equation}\label{eq6}
\psi \phi'(\mathcal{A}^{\sigma_m}(t_l))=\psi \phi'(t_l) \quad (l=1, \ldots,n,\, m=1, \ldots, n-1),
\end{equation}
which implies $\psi\phi'(\mathcal{A}^b(u_j))=\psi \phi'(u_j)$. 
Since $\phi'(\mathcal{A}^b(u_j))=\phi(A^b(x_j))$ and $\phi'(u_j)=\phi(x_j)$, we see that 
\begin{equation}\label{eq5}
\psi \phi(A^b(x_j))=\psi \phi(x_j) \quad (j=1, \ldots, 2g).  
\end{equation}
Since $\psi\phi(x_j)=1$ and $\psi \phi(z)=t$, the left side (respectively, the right side) of (\ref{eq5}) is the sum of indices of $z$ in $A^b(x_j)$ (respectively, in $x_j$). Hence we see that $\nu_1+ \ldots +\nu_{n_i}=0$ for $r_i=x_{i'} (A^b(x_{i'}))^{-1}$ ($i=2g+1, \ldots, 4g, i'=i-2g$), which is the required formula. Similarly, the equation (\ref{eq6}) implies $\psi \phi (A^b(z))=\psi \phi (z)$, and we have the required formula for $r_{4g+1}$. 
\end{proof}

\begin{proof}[Proof of Lemma \ref{lem4}]
(1) By Fox calculus, 
\begin{equation}\label{eq6-11}
\frac{\partial r_i}{\partial x_j}=z^{\nu_1} \frac{\partial w_{i1}}{\partial x_j}+z^{\nu_1} w_{i1}z^{\nu_2} \frac{\partial w_{i2}}{\partial x_j} +z^{\nu_1} w_{i1}z^{\nu_2} w_{i2} z^{\nu_3} \frac{\partial w_{i3}}{\partial x_j} +\cdots, 
\end{equation}
for $i=1, \ldots, 4g+1$, and $j=1, \ldots, 2g$. 
We compute $\psi\phi \left( \partial w_{ik}/\partial x_j\right)$, as follows.
By Fox calculus, for a word $w=x_{j_1}^{\epsilon_1} x_{j_2}^{\epsilon_2} \cdots x_{j_l}^{\epsilon_{l}}$ $(\epsilon_1,\ldots, \epsilon_{l} \in \{+1, -1\})$, 
\begin{equation}\label{eq6a}
\frac{\partial w}{\partial x_j}=\epsilon_1 \delta_{j j_1} x_{j_1}^{\frac{\epsilon_1-1}{2}}+\epsilon_2 \delta_{j j_2} x_{j_1}^{\epsilon_1} x_{j_2}^{\frac{\epsilon_2-1}{2}}+\epsilon_3 \delta_{j j_3} x_{j_1}^{\epsilon_1} x_{j_2}^{\epsilon_2} x_{j_3}^{\frac{\epsilon_3-1}{2}}+\cdots,
\end{equation}
where $\delta_{j j_1}, \delta_{j j_2}, \ldots$ are the Kronecker delta (\cite[Exercise 12.5 (2)]{Kamada02}). 
By the argument in the proof of Lemma \ref{lem3}, $\psi\phi(x_j)=1$ $(j=1,\ldots, 2g)$ and $\psi\phi(z)=t$. 
Together with the equation $\psi\phi(x_j)=1$ $(j=1,\ldots, 2g)$, the formula (\ref{eq6a}) implies $\psi\phi \left( \partial w_{ik}/\partial x_j\right)=\mu_{ik}^j$ ($i=1, \ldots, 4g+1$, $j=1, \ldots, 2g$, $k=1, \ldots, \nu_i$). Since $\psi \phi (w_{ik})=1$ ($i=1, \ldots, 4g+1$, $k=1, \ldots, n_i$) from $\psi\phi(x_j)=1$ $(j=1,\ldots, 2g)$, together with $\psi\phi(z)=t$, by (\ref{eq6-11}), we see that 
\begin{equation}\label{eq7}
\psi\phi \left( \frac{\partial r_i}{\partial x_j}\right)=t^{\nu_1} \mu_{i1}^j +t^{\nu_1+\nu_2}\mu_{i2}^j +t^{\nu_1+\nu_2+\nu_3}\mu_{i3}^j +\cdots.
\end{equation}
By putting $t=-1$ into (\ref{eq7}), we have the required result.

(2) By Fox calculus, 
\begin{eqnarray*}
\frac{\partial r_i}{\partial z}&=&\frac{\partial z^{\nu_1}}{\partial z}+z^{\nu_1} w_{i1} \frac{\partial z^{\nu_2}}{\partial z} +z^{\nu_1} w_{i1}z^{\nu_2} w_{i2} \frac{\partial z^{\nu_3}}{\partial z} +\cdots \\
&& +z^{\nu_1} w_{i1}z^{\nu_2} w_{i2} \cdots z^{\nu_{n_i-1}} w_{i (n_i-1)} \frac{\partial z^{\nu_{n_i}}}{\partial z},  
\end{eqnarray*}
for $i=1, \ldots, 4g+1$. 

Since $\psi\phi (w_{ik})=1$ ($i=1, \ldots, 4g+1$, $k=1,\ldots, n_i$), $\psi\phi\left( \partial r_i/\partial z \right)$ equals 

\begin{eqnarray*}
\psi\phi \left( \frac{\partial z^{\nu_1}}{\partial z}+z^{\nu_1}  \frac{\partial z^{\nu_2}}{\partial z} +z^{\nu_1+\nu_2}  \frac{\partial z^{\nu_3}}{\partial z} +\cdots \right.
 \left. +z^{\nu_1+\nu_2+\cdots \nu_{n_i-1}} \frac{\partial z^{\nu_{n_i}}}{\partial z} \right), 
\end{eqnarray*}
which is $\psi\phi \left( \partial z^{\nu_1+\nu_2+\cdots +\nu_{n_i}}/\partial z\right)$. 
Since $\nu_1+\nu_2+\cdots+\nu_{n_i}=0$ by Lemma \ref{lem3}, 
\[
\psi\phi \left( \frac{\partial r_i}{\partial z}\right)=\psi\phi \left( \frac{\partial z^0}{\partial z}\right)=\psi\phi \left( \frac{\partial 1}{\partial z}\right)=0, \]
which is the required result. 
\end{proof}

\section{Proof of Theorem \ref{thm2-1}}\label{sec-proof}
In this section, we show Theorem \ref{thm2-1}. 
Our argument greatly relies on that due to Lin \cite{Lin}. 

\begin{proof}[Proof of Theorem \ref{thm2-1}]
We consider the presentation (\ref{eq-*}) of $\pi_1(S^4 \backslash \mathcal{S}_n(a,b))$: 
\begin{equation}\label{eq-5-27}
\left\langle 
\begin{matrix}  x_1, x_2, \ldots, x_{2g}, z \end{matrix}
 \left| 
\begin{matrix}
 \ z\alpha_j z^{-1}=\beta_j,&\\
 \ x_j=A^b(x_j),& \\
 z=A^b(z) &\,(j=1,2,\ldots, 2g)
\end{matrix} \right. 
\right\rangle,
\end{equation}
where $x_1, \ldots, x_{2g}$ are generators of $\pi_1(D^3 \times \{  *\} \backslash S)=F(x_1, \ldots, x_{2g})$ for the associated Seifert surface $S$ of $\hat{a} \subset D^3 \times \{ * \}$, and 
$\alpha_j$ and $\beta_j$ are the induced words associated with $S$ and its spine, and $z$ is represented by a meridian, and  
$A^b(x_j)$ and $A^b(z)$ are words associated with Artin's automorphism $\mathcal{A}^b$. 

Since $x_j$ $(j=1,\ldots, 2g)$ is represented by a closed path in $D^3 \times \{*\} \backslash S \times [-0.5, 0.5]$, $\psi\phi(x_j)=1$ $(j=1,\ldots, 2g)$, and hence $x_j$ $(j=1,\ldots, 2g)$ is an element of the commutator subgroup of $\pi_1(S^4 \backslash \mathcal{S}_n(a,b))$. 
Let $\rho: \pi_1(S^4 \backslash \mathcal{S}_n(a,b)) \to SU(2)$ be an irreducible metabelian representation. 
Since $\rho$ is metabelian, and since we consider up to conjugation, by simultaneous diagonalization we arrange that  
each $\rho(x_j)$ $(j=1,\ldots, 2g)$ is a diagonal matrix. We put 
\begin{equation}\label{eq-xj}
\rho(x_j)=\begin{pmatrix}
\lambda_j & 0 \\
0 & \bar{\lambda_j}
\end{pmatrix}
=\begin{pmatrix}
\lambda_j & 0 \\
0 & {\lambda}_j^{-1}
\end{pmatrix}, \quad j=1,\ldots, 2g, 
\end{equation}
where $\lambda_j$ $(j=1,\ldots, 2g)$ is an element in the unit circle in $\mathbb{C}$. 
Since $\rho$ is metabelian, by Lemma \ref{lem6}, $\rho(zx_jz^{-1})$ $(j=1,\ldots, 2g)$ is also a diagonal matrix. 
Together with the condition that $\rho$ is irreducible, by Lemma \ref{lem6a}, by taking conjugation if necessary, we have  
\begin{equation}\label{eq-z}
\rho(z)=\begin{pmatrix}
 0 & 1  \\
 -1 & 0
 \end{pmatrix}. 
 \end{equation}
We see that $\rho(z^2)=-E$, where $E$ is the unit matrix. Hence, for an integer $\nu$, 
 \[
 \rho(z^\nu x_j z^{-\nu})=\begin{cases}\begin{pmatrix}
{\lambda}_j^{-1} & 0 \\
0 & {\lambda}_j
\end{pmatrix} & \text{if $\nu$ is odd}\\
\\
\begin{pmatrix}
{\lambda}_j & 0 \\
0 & {\lambda}_j^{-1}
\end{pmatrix} & \text{ if $\nu$ is even, }\\
\end{cases}
 \]
that is, 
\[
 \rho(z^\nu x_j z^{-\nu})=\begin{pmatrix}
{\lambda}_j^{(-1)^\nu} & 0 \\
0 & {\lambda}_j^{-(-1)^\nu}
\end{pmatrix} \quad  (\nu \in \mathbb{Z}),
\]
for $j=1,\ldots, 2g$.
 We put these matrices into relations of (\ref{eq-5-27}), which are $r_1, \ldots, r_{4g+1}$ with the words of (\ref{rel2}). 
The relation $r_i$ $(i=1,\ldots, 4g+1)$ is written as follows:
\begin{eqnarray*}
r_i& =&(z^{\nu_1}w_{i1} z^{-\nu_1})(z^{\nu_1+\nu_2}w_{i2}z^{-(\nu_1+\nu_2)}) \cdots \\
&& \ \cdot (z^{\nu_1+ \cdots +\nu_{n_i}} w_{i n_i}z^{-(\nu_1+ \cdots +\nu_{n_i})})z^{\nu_1 +\cdots +\nu_{n_i}}. 
\end{eqnarray*}
Since $\nu_1+\cdots +\nu_{n_i}=0$ by Lemma \ref{lem3}, 
\begin{eqnarray*}
r_i =(z^{\nu_1}w_{i1} z^{-\nu_1})(z^{\nu_1+\nu_2}w_{i2}z^{-(\nu_1+\nu_2)}) \cdots 
(z^{\nu_1+ \cdots +\nu_{n_i}} w_{i n_i}z^{-(\nu_1+ \cdots +\nu_{n_i})}). 
\end{eqnarray*}
Recalling that $\mu_{ik}^j$ is the sum of indices of $x_j$ in the word $w_{ik}$ in $r_i$, and seeing the diagonal elements of the matrix, the relation $\rho(r_i)=E$ is equivalent to:
  \begin{equation}\label{eq-2}
\lambda_1^{a_{i 1}} \lambda_2^{a_{i2}} \cdots \lambda_{2g}^{a_{i (2g)}}=1, 
   \end{equation}
where 
\[
a_{i j}=(-1)^{\nu_1} \mu_{i 1}^j+(-1)^{\nu_1+\nu_2} \mu_{i 2}^j+ \cdots+(-1)^{\nu_1+\nu_2+\cdots +\nu_{n_i}} \mu_{i n_i}^j
\]
 $(i=1, \ldots, 4g+1, j=1, \ldots, 2g)$. 

Put $A=(a_{ij})$ $(i=1, \ldots, 4g+1, j=1, \ldots, 2g)$. 
By Lemma \ref{lem4} (1) and (2), $a_{i j}=\psi\phi \left(  \partial r_i/\partial x_j \right)|_{t=-1}$, $\psi\phi \left(  \partial r_i/\partial z \right)=0$; 
and it follows that 
\begin{equation}\label{eqA}
\begin{pmatrix}A \ \mathbf{o}\end{pmatrix}=A(t)|_{t=-1},
\end{equation}
where $A(t)$ is the Alexander matrix, and $\mathbf{o}$ is the zero column vector with $4g+1$ elements.
Since $\lambda_j$ is an element of the unit circle in $\mathbb{C}$, $\lambda_j=e^{\sqrt{-1}\eta_j}$ for some $\eta_j \in \mathbb{R}$ ($j=1,\ldots, 2g$); and   
the equation $(\ref{eq-2})$ is equivalent to: 
\[
a_{i1}\eta_1+a_{i 2}\eta_2+\cdots+a_{i (2g)} \eta_{2g}=0 \pmod {2\pi}, 
\]
$(i=1,\ldots, 4g+1)$, and the system of these equations is described by: 
\begin{equation}\label{eq-a}
A \vec{\eta}=\mathbf{o} \pmod{2\pi}, 
\end{equation}
where $\vec{\eta}$ is a column vector $\vec{\eta}=( \eta_1, \ldots, \eta_{2g})^T \in \mathbb{R}^{2g}$. 
By elementary transformations for integer matrices, $A$ is transformed to the Smith normal form. 
Now, the knot determinant of $F=\mathcal{S}_n(a,b)$, denoted by $\Delta(F)|_{t=-1}$, is the nonnegative generator of the $\mathbb{Z}$-ideal $\{ f(-1) \mid f(t) \in \Delta(F)\}$, where $\Delta(F)$ is the Alexander ideal. Further, $\Delta(F)|_{t=-1}$ is odd, since the knot group of $F$ is the normal subgroup of the classical knot group of $\hat{a}$, whose knot determinant is odd; see also \cite[Proposition 5.8]{Joseph}.  
Hence (\ref{eqA}) implies that the Smith normal form of $A$ is as follows: 
\begin{equation*}
PAQ=\begin{pmatrix}
d_1 & 0 & \cdots & 0\\
0 & d_2 &\\
0 & 0 & \ddots \\
\vdots &&  & d_{2g} \\
0 & 0 & \cdots & 0\\
& \vdots\\
0 & 0 & \cdots & 0
\end{pmatrix},
\end{equation*}
where $P$ and $Q$ are unimodular matrices associated with the transformations, and $d_1, \ldots, d_{2g}$ are positive odd integers such that $d_{j+1}$ is a multiple of $d_{j}$ $(j=1, \ldots, 2g-1$), and  
\begin{equation*}
d_1 \cdots d_{2g}=\Delta(F)|_{t=-1}.  
\end{equation*}
Solutions of the equation $PAQ \vec{\eta}{\,'}=\mathbf{o} \pmod{1}$ $(\vec{\eta}{\,'} \in \mathbb{R}^{2g})$ are given by 
\begin{equation}\label{eq-0628-2021}
\vec{\eta}{\,'}=\left(\frac{n_1}{d_1}, \ldots, \frac{n_{2g}}{d_{2g}}\right)^T \pmod{1}, 
\end{equation}
where $n_j =0, 1,\ldots, d_j-1$ $(j=1,\ldots, 2g)$. 
Hence solutions of (\ref{eq-a}) are given by $(2 \pi)Q\vec{\eta}{\,'}$ for $\vec{\eta}{\,'}$ of (\ref{eq-0628-2021}), 
and the 
number of solutions is $d_1\cdots d_{2g}=\Delta(F)|_{t=-1}$. 

For the trivial solution of (\ref{eq-a}), the set of associated matrices $\{\rho(x_1), \ldots, \rho(x_{2g})\}$ is $\{E\}$; hence the image of $\rho$ is generated by $\rho(z)$, and $\rho$ is reducible, which is a contradiction. For any non-trivial solution of (\ref{eq-a}), since each $d_j$ is odd $(j=1, \ldots, 2g)$, $\{\rho(x_1), \ldots, \rho(x_{2g})\}$ is neither $\{E\}$ nor $\{-E\}$ nor $\{E, -E\}$; hence $\rho(x_j) \neq \pm E$ for some $j$, and $\rho(x_j)$ and $\rho(z)$ don't have a common eigenvector, so $\rho$ is irreducible. 
Hence irreducible representations are given by non-trivial solutions of (\ref{eq-a}). 
We consider a non-trivial solution $\vec{\eta}_0=(\eta_1, \ldots, \eta_{2g})$ of (\ref{eq-a}). Then, $-\vec{\eta}_0=(-\eta_1, \ldots, -\eta_{2g})$ is also a solution.  The associated matrices $\rho(x_1), \ldots, \rho(x_{2g})$ are 
\begin{equation*}\label{eq-m}
\begin{cases} \begin{pmatrix}
\lambda_{1} & 0\\
0 & \lambda_1^{-1}
\end{pmatrix}, \ldots, 

\begin{pmatrix}
\lambda_{2g} & 0\\
0 & \lambda_{2g}^{-1}
\end{pmatrix} & \text{for $\vec{\eta}_0=(\eta_1, \ldots, \eta_{2g})$} \\ 
\\

\begin{pmatrix}
\lambda_{1}^{-1} & 0\\
0 & \lambda_1
\end{pmatrix}, \ldots, 
\begin{pmatrix}
\lambda_{2g}^{-1} & 0\\
0 & \lambda_{2g}
\end{pmatrix} & \text{for $-\vec{\eta}_0=(-\eta_1, \ldots, -\eta_{2g})$}

\end{cases}
\end{equation*}
($\lambda_j=e^{\sqrt{-1} \eta_j}$, $j=1,\ldots, 2g$), and they are simultaneously conjugate by $\rho(z)$ (\ref{eq-z}).  
And since the diagonalization of $\rho(x_j)$ $(j=1, \ldots, 2g)$ is 
\[
\begin{pmatrix}
\lambda_{j} & 0\\
0 & \lambda_{j}^{-1}
\end{pmatrix} \text{or} 
\begin{pmatrix}
\lambda_{j}^{-1} & 0\\
0 & \lambda_{j}
\end{pmatrix}, 
\]
representations associated with solutions other than $-\vec{\eta}_0$ are not conjugate to that associated with $\vec{\eta}_0$. 
Thus we see that the number of conjugacy classes of irreducible metabelian $SU(2)$-representations is half of $(\Delta(F)|_{t=-1})-1$, which is the required result.
\end{proof}

\begin{lemma}\label{lem6}
We use the notations in the proof of Theorem \ref{thm2-1}. 
The $SU(2)$-representation $\rho$ is metabelian, and $x_j$ is an element of the commutator subgroup, and $\rho(x_j)$ is a diagonal matrix $(\ref{eq-xj})$.  
Then, $\rho(zx_jz^{-1})$ is also a diagonal matrix $(j=1,\ldots, 2g)$. 
\end{lemma}

\begin{proof}
Put 
\[
\rho(zx_jz^{-1}x_j^{-1})=\begin{pmatrix}
\alpha & -\bar{\beta}\\
\beta & \bar{\alpha}
\end{pmatrix} \in SU(2),
\]
where $\alpha$, $\beta$ are elements in $\mathbb{C}$ with $|\alpha|^2+|\beta|^2=1$. 
If $\beta=0$, then $\rho(zx_jz^{-1}x_j^{-1})$ is a diagonal matrix. Since $\rho(x_j)$ is a diagonal matrix, it follows that $\rho(z x_j z^{-1})=\rho(z x_j z^{-1}x_j^{-1})\rho(x_j)$ is also a diagonal matrix. 
We consider the case $\beta \neq 0$. 
Since $x_j$ is an element of the commutator subgroup and $\rho$ is metabelian, $\rho(zx_jz^{-1}x_j^{-1}) \rho(x_j)=\rho(x_j) \rho(zx_jz^{-1}x_j^{-1})$; hence
\begin{eqnarray*}
\begin{pmatrix}
\alpha & -\bar{\beta}\\
\beta & \bar{\alpha} 
\end{pmatrix}
\begin{pmatrix}
\lambda_j & 0 \\
0 & \bar{\lambda_j} 
\end{pmatrix}
=\begin{pmatrix}
\lambda_j & 0 \\
0 & \bar{\lambda_j} 
\end{pmatrix}
\begin{pmatrix}
\alpha & -\bar{\beta}\\
\beta & \bar{\alpha} 
\end{pmatrix},\\
\end{eqnarray*}
thus 
\[
\begin{pmatrix}
\alpha\lambda_j & -\bar{\beta} \bar{\lambda_j}\\
\beta\lambda_j & \bar{\alpha} \bar{\lambda_j}
\end{pmatrix}
=
\begin{pmatrix}
\alpha\lambda_j & -\bar{\beta} \lambda_j\\
\beta\bar{\lambda_j} & \bar{\alpha} \bar{\lambda_j}
\end{pmatrix}.
\]
This implies that $\lambda_j=\bar{\lambda_j}$, and it follows that $\lambda_j=\pm 1$; hence $\rho(x_j)=\pm E$, where $E$ is the unit matrix. Then, $\rho(z x_j z^{-1})=\rho(x_j)$, which is a diagonal matrix. 
\end{proof}

\begin{lemma}\label{lem6a}
We use the notations in the proof of Theorem \ref{thm2-1}. 
Then, by taking conjugation, we have $\rho(z)=\begin{pmatrix}
0 & 1 \\
-1 & 0
\end{pmatrix}$. 
\end{lemma}

\begin{proof}
Put \[
\rho(z)=\begin{pmatrix}
\alpha & -\bar{\beta}\\
\beta & \bar{\alpha}
\end{pmatrix},
\]
where $\alpha$, $\beta \in \mathbb{C}$ with $|\alpha|^2+|\beta|^2=1$. 
Since each $\rho(z x_j z^{-1})$ $(j=1,\ldots, 2g)$ is a diagonal matrix by Lemma \ref{lem6}, the result of the computation
\begin{eqnarray*}
\rho(z x_j z^{-1})
&=&
\begin{pmatrix}
\alpha & -\bar{\beta}\\
\beta & \bar{\alpha}
\end{pmatrix}
\begin{pmatrix}
\lambda_j & 0 \\
0 & \bar{\lambda_j} 
\end{pmatrix}
\begin{pmatrix}
\bar{\alpha} & \bar{\beta}\\
-\beta & \alpha
\end{pmatrix}\\
&=& 
\begin{pmatrix}
|\alpha|^2 \lambda_j+|\beta|^2 \bar{\lambda_j} & \alpha\bar{\beta}(\lambda_j-\bar{\lambda_j}) \\
-\bar{\alpha}{\beta}(\bar{\lambda_j}-{\lambda}_j) & 
|\alpha|^2 \bar{\lambda_j}+|\beta|^2 {\lambda}_j
\end{pmatrix}
\end{eqnarray*}
implies that $\alpha=0$, or $\beta=0$, or $\lambda_j=\bar{\lambda_j}$ for $j=1,\ldots, 2g$. 
If $\beta=0$, then $\rho(z)$ is a diagonal matrix; since each $\rho(x_j)$ is also a diagonal matrix, we see that $\rho$ is reducible, which is a contradiction. If $\lambda_j=\bar{\lambda_j}$ for each $j$, then $\rho(x_j)=E$ or $-E$, and the image of $\rho$ is generated by $\rho(z)$ or $\pm \rho(z)$; hence $\rho$ is reducible, which is a contradiction. If $\alpha=0$, then \[
\rho(z)=\begin{pmatrix}
0 & -\bar{\beta}\\
\beta & 0
\end{pmatrix}, \]
where $\beta$ is in the unit circle in $C$. 
Put $\beta=e^{i\theta}$ $(\theta \in \mathbb{R})$, and take 
\[
U=\begin{pmatrix}
ie^{-\frac{i\theta}{2}} & 0 \\
0 & -ie^{\frac{i\theta}{2}}
\end{pmatrix} \in SU(2). \]
Since \[
U^{-1} \rho(z) U=\begin{pmatrix}
0 & 1\\
-1 & 0
\end{pmatrix},
\]
by taking conjugation by $U$, we have new matrices $\rho(x_j)$ $(j=1,\ldots, 2g)$, which are diagonal, and we have the new $\rho(z)$ that is
\[
\rho(z)=\begin{pmatrix}
0 & 1\\
-1 & 0
\end{pmatrix}.
\]
\end{proof}

\section{Quandle colorings and proofs of Theorems \ref{thm2-2} and \ref{thm2-3} and Corollary \ref{cor2-4}}\label{sec7}

In this section, we review quandle colorings \cite{CKS, EN, HK, Joyce, P}, and we show Theorems \ref{thm2-2} and \ref{thm2-3} and Corollary \ref{cor2-4}. We use arguments related to \cite{P}. 
Recall that we treat classical knots or surface-knots that are oriented. 

\subsection{Quandles}

 A {\it quandle} is a set $X$ equipped with a binary operation $*: X \times X \to X$ satisfying the following axioms. 
 
 \begin{enumerate}
 \item (Idempotency)
 For any $x \in X$, $x*x=x$. 
 
 \item (Right invertibility)
 For any $y,z \in X$, there exists a unique $x \in X$ such that $x*y=z$. 
 
 \item (Right self-distributivity)
 For any $x,y,z \in X$, $(x*y)*z=(x*z)*(y*z)$. 
 \end{enumerate}
 
Let $K$ be a classical knot and let $F$ be a surface-knot, respectively. 
 We review quandle colorings for $K$ or $F$. 
Let $D$ be a diagram of $K$ or $F$. 
Here, we obtain a diagram by a method as follows. 
We take the image of $K$ or $F$ by a generic projection to $\mathbb{R}^2$ or $\mathbb{R}^3$. In order to equip the image with crossing information, we break each under-arc or under-sheet into two pieces around each crossing or double point curve. A {\it diagram} is the set consisting of resulting arcs or compact surfaces, which are called {\it over-arcs/under-arcs}, or simply {\it arcs} for $K$, or,  {\it over-sheets/under-sheets}, or simply {\it sheets} for $F$. 
 For a diagram $D$, let $B(D)$ be the set of arcs or sheets of $D$. 
An {\it $X$-coloring} of $D$ is a map $C: B(D) \to X$ satisfying the coloring rule around each crossing or double point curve as shown in Fig. \ref{fig-quandle}. For an $X$-coloring $C$, the image of an arc or sheet by $C$ is called a {\it color}. 

\begin{figure}[ht]
\includegraphics[height=4cm]{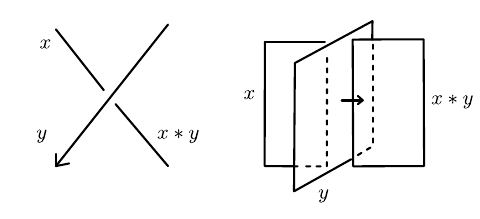}
\caption{The quandle coloring rule, where $x,y \in X$. We denote the orientation of the over-sheet by its normal vector. The orientation of under-arcs or under-sheets is arbitrary.}\label{fig-quandle}
\end{figure}

Let $p>1$ be an integer. A {\it dihedral quandle} $R_p$ is a set $R_p=\mathbb{Z}/p\mathbb{Z}$ with a binary operation 
\[
x*y=2y-x,
\]
where $x,y \in R_p$. 
And an {\it Alexander quandle} is a module over $\Lambda=\mathbb{Z}[t, t^{-1}]$ with a binary operation 
\[
x *y=tx+(1-t)y, 
\]
where $x, y$ are elements of the Alexander quandle. 
In particular, we consider an Alexander quandle $\Lambda$. 
The dihedral quandle $R_p$ is obtained from the Alexander quandle $\Lambda$ by putting $t=-1$ and taking a quotient modulo $p$. 

\subsection{The Alexander matrix associated with $\Lambda$-colorings}
Let $K$ and $F$ be a classical knot and a surface-knot, respectively. 
Let $D$ be a diagram. 
For an Alexander quandle $\Lambda$, we consider $\Lambda$-colorings of $D$. 
Let $m$ be the number of arcs or sheets of $D$. 
Let $x_1, x_2, \ldots,x_m$ be variables in $\Lambda$ such that each $x_j$ corresponds to the $j$th arc or sheet $(j=1, \ldots, m)$. 
At each crossing or double point curve, we have a quandle relation $x_i*x_j=x_k$ for some $i,j,k \in \{1, \ldots, m\}$; since we consider an Alexander quandle relation, we have an equation 
$tx_i+(1-t)x_j=x_k$, that is, $tx_i+(1-t)x_j-x_k=0$. Let $n$ be the number of crossings or double point curves of $D$. Then, the system of equations associated with crossings or double point curves is described by  $A'(t) \cdot \mathbf{x}=\mathbf{o}$, where $\mathbf{x}=(x_1, x_2, \ldots, x_m)^T$, and $A'(t)$ is an $(n,m)$-matrix over $\Lambda$; and each solution $\mathbf{x} \in \Lambda^m$ presents each $\Lambda$-coloring. We call $A'(t)$ the {\it Alexander matrix associated with $\Lambda$-colorings} of $D$. 
We remark that $A'(t)$ is \lq\lq almost'' equivalent to the Alexander matrix obtained from the Wirtinger presentation associated with $D$ (Lemma \ref{lem-wirtinger}).

\begin{lemma}\label{lem-wirtinger}
Let $K$ and $F$ be a classical knot and a surface-knot, respectively. Let $D$ be a diagram. 
Let $G$ be the Wirtinger presentation of the knot group associated with $D$. 
Let $A(t)$ (respectively, $A'(t)$) be the Alexander matrix associated with $G$ (respectively, $\Lambda$-colorings of $D$). 
Then, $A'(t)$ is equivalent to $A(t^{-1})$ by operations of presentation matrices. 
\end{lemma}

Before the proof, 
we review a Wirtinger presentation of the knot group of $K$ or $F$. Let $m$ be the number of arcs or sheets of $D$. 
Then, the knot group of $K$ or $F$ has the presentation as follows.  
Let $\mathbf{p}$ be the projection to $\mathbb{R}^2$ or $\mathbb{R}^3$ associated with $D$. We take generators  $x_1, \ldots, x_m$ such that each $x_j$ $(j=1, \ldots, m)$  
is presented by a closed path that starts from a base point, goes straight to the preimage by $\mathbf{p}$ of the $j$th arc or sheet, encircles it once in a positive direction, and returns straight to the base point; here, we regard the preimage $\mathbf{p}^{-1}(\alpha)$ of a sheet $\alpha$ as the product of an arc in a 3-space and an interval in the fourth direction, and the closed path goes around an arc which forms $\mathbf{p}^{-1}(\alpha)$. 
At each crossing or double point curve $c$ as shown in Fig. \ref{fig-wirtinger}, we take a relation $x_i^{-1} x_j x_k x_j^{-1}$ (respectively, $x_i x_j x_k^{-1} x_j^{-1}$) if $c$ is a positive crossing or the product of a positive crossing and an interval (respectively, a negative crossing or the product of a negative crossing and an interval) as shown in the left (respectively, the right) of Fig. \ref{fig-wirtinger}. Let $n$ be the number of crossings or double point curves of $D$, and let $r_1, \ldots, r_n$ be relations associated with crossings or double point curves. The {\it Wirtinger presentation of the knot group associated with $D$} is the presentation 
\begin{equation}\label{eq-wirtinger}
\langle x_1, \ldots, x_m \mid r_1, \ldots, r_n \rangle. 
\end{equation}

\begin{figure}[ht]
\includegraphics[height=4.3cm]{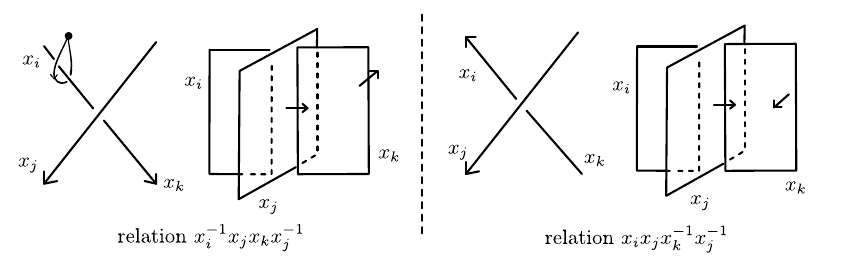}
\caption{Relations of the Wirtinger presentation, where $x_i,x_j, x_k$ are generators. We denote the orientation of each sheet by its normal vector, and we abbreviate closed paths presenting $x_i, x_j, x_k$ except that presenting $x_i$ in the leftmost figure.}\label{fig-wirtinger}
\end{figure}

\begin{proof}[Proof of Lemma \ref{lem-wirtinger}]

Let $G$ be the Wirtinger presentation of the knot group associated with $D$, given by (\ref{eq-wirtinger}).  
Let $\phi: F(x_1, \ldots, x_m) \to G$ be the projection, and let 
$\psi: G \to G/[G,G]=\langle t \rangle$ 
be the abelianization map. We have $\psi \phi(x_l)=t$ $(l=1,\ldots, m)$. 
We denote by the same notation $\phi$ and $\psi$ the induced ring homomorphisms of group rings over $\mathbb{Z}$. 
Let $f: \Lambda=\mathbb{Z}[t, t^{-1}] \to \Lambda$ be the  homomorphism induced from $f(t)=t^{-1}$. 

Each row of $A'(t)$, which will be denoted by $\mathbf{a}'=(a_1', \ldots, a_m')$, is associated with an equation 
\begin{equation}\label{eq-0616}
tx_i+(1-t)x_j-x_k=0
\end{equation}
 for some $i,j,k \in \{1,\ldots, m\}$. 
The corresponding Wirtinger relation is either $x_i^{-1} x_j x_kx_j^{-1}$ or $x_i x_j x_k^{-1}x_j^{-1}$. Denote the first relation $x_i^{-1} x_j x_kx_j^{-1}$ by $r$. Since $x_i x_j x_k^{-1}x_j^{-1}=x_i r^{-1} x_i^{-1}$, the second relation is the same with $r$. 
We denote by 
$\mathbf{a}=(a_1, \ldots, a_m)$ the row of $A(t)$ associated with $r$. 

For all possible cases, the left side of the equation (\ref{eq-0616}) is as follows.
\begin{equation}\label{eq-0616-02}
\begin{cases} tx_i+(1-t)x_j-x_k & \text{if $i\neq j$, $j \neq k$, $k \neq i$}\\
x_i-x_k & \text{if $i=j \neq k$}\\
tx_i-tx_j & \text{if $j=k \neq i$}\\
(1-t)x_j-(1-t)x_k & \text{if $k=i \neq j$} \\
0 & \text{if $i=j=k$}.
\end{cases}
\end{equation}
First 
we consider the case when $i,j,k$ are mutually distinct.  Then,  
\begin{equation} 
a_l'=
\begin{cases}t & \text{if $l=i$} \\
1-t & \text{if $l=j$} \\
-1 & \text{if $l=k$}\\
0 & \text{otherwise.}
\end{cases}
\end{equation}

We calculate $\mathbf{a}=(a_1, \ldots, a_m)$ and $f(\mathbf{a})=(f(a_1), \ldots, f(a_m))$. 
Recall that $r=x_i^{-1} x_j x_k x_j^{-1}$. 
By Fox calculus, 
\begin{eqnarray*}
 \frac{\partial r}{\partial x_l} =\begin{cases} -x_i^{-1}  & \text{if $l=i$}\\
x_i^{-1}-x_i^{-1}x_j  x_kx_j^{-1}& \text{if $l=j$}\\
x_i^{-1} x_j & \text{if $l=k$}\\
0  & \text{otherwise}.
\end{cases}
\end{eqnarray*}
Hence 
\begin{eqnarray*}
a_l=\psi\phi\left( \frac{\partial r}{\partial x_l} \right)=\begin{cases} -t^{-1} & \text{if $l=i$}\\
t^{-1}-1& \text{if $l=j$}\\
1 & \text{if $l=k$}\\
0  & \text{otherwise}.
\end{cases}
\end{eqnarray*}
Thus 
\begin{eqnarray*}
f(a_l)=\begin{cases} -t & \text{if $l=i$}\\
t-1& \text{if $l=j$}\\
1 & \text{if $l=k$}\\
0  & \text{otherwise}.
\end{cases}
\end{eqnarray*}
Hence, in this case, $\mathbf{a}'$ is multiplication of $f(\mathbf{a})$ by $-1$. 

Similarly, for the other cases when $i,j,k$ are not mutually distinct, 
we see that $\mathbf{a}'$ is multiplication of $f(\mathbf{a})$ by a unit of $\Lambda$. 

Thus $A'(t)$ is related to $A(f(t))=A(t^{-1})$ by a sequence of operations each of which is multiplication of a row by a unit of $\Lambda$; hence $A'(t)$ is equivalent to $A(t^{-1})$. 
\end{proof}

\subsection{Fox's p-colorings}
Let $p>1$ be an integer. 
For a dihedral quandle $R_p$, 
we call an $R_p$-coloring a {\it $p$-coloring}. Let $K$ and $F$ be a classical knot and a surface-knot, respectively. Let $D$ be a diagram. 
We denote by $\mathrm{Col}_p(D)$ the set of $p$-colorings of $D$. 
We denote by $|\mathrm{Col}_p(K)|$ or $|\mathrm{Col}_p(F)|$ the number of elements of $\mathrm{Col}_p(D)$, which is invariant under diagrams of $K$ or $F$. We say that a $p$-coloring is {\it non-trivial} if colors of arcs or sheets contain at least two distinct colors. We say that a $p$-coloring is {\it non-degenerate} if colors of arcs or sheets generate $R_p$ by the binary operation $*$, and we say that a knot $K$ or $F$, or a diagram $D$, is {\it $p$-colorable} if $D$ admits a non-degenerate $p$-coloring. 
A $p$-coloring is {\it trivial} (respectively, {\it degenerate}) if it is not non-trivial (respectively, not non-degenerate). 

Let $A(t)$ be the Alexander matrix associated with $\Lambda$-colorings of $D$. 
For the equation $A(-1) \cdot \mathbf{x}=\mathbf{o}$, we say that $\mathbf{x}_0$ is a {\it solution modulo $p$} if each element of $\mathbf{x}_0$ is valued in $\mathbb{Z}/p \mathbb{Z}$ and $A(-1) \cdot \mathbf{x}_0=\mathbf{o} \pmod{p}$. 
Since the dihedral quandle $R_p=\mathbb{Z}/p\mathbb{Z}$ is obtained from an Alexander quandle $\Lambda$ by putting $t=-1$ and taking modulo $p$, we have the following lemma (Lemma \ref{lem-0613-5}). 

\begin{lemma}\label{lem-0613-5}
Let $K$ and $F$ be a classical knot and a surface-knot, respectively. Let $D$ be a diagram. 
Let $A(t)$ be the Alexander matrix associated with $\Lambda$-colorings of $D$, and  
let $m$ be the number of columns of $A(t)$. 
Let $p>1$ be an integer.
Then, the set of solutions modulo $p$ of $A(-1) \cdot \mathbf{x}=\mathbf{o}$ has one-to-one correspondence with the set of $p$-colorings of $D$, where each $j$th element $x_j$ of a solution $\mathbf{x}=(x_1, \ldots, x_m)^T$ corresponds to the color of the $j$th arc or sheet of $D$ $(j=1, \ldots, m)$. 
\end{lemma}

From now on, we investigate properties of $p$-colorings. 

\begin{lemma}\label{lem920-5}
Let $K$ and $F$ be a classical knot and a surface-knot, respectively. Let $D$ be a diagram. Let $p>1$ be an integer. Assume that $D$ is $p$-colorable. Then, for a given arc or sheet $\alpha$ and a given color $x$, $D$ admits a non-degenerate $p$-coloring $C$ such that 
$C(\alpha)=x$. 
\end{lemma}

\begin{proof}
Since $D$ is $p$-colorable, $D$ admits a non-degenerate $p$-coloring, which will be denoted by $C_0$. 
Then, define $C$ by $C(\beta)=C_0(\beta)-C_0(\alpha)+x \pmod{p}$, for any arc or sheet $\beta$ of $D$. Then, $C$ is a $p$-coloring with $C(\alpha)=x$. Further, since the set generated by colors by $C_0$ using the dihedral quandle operation $*$ has one-to-one correspondence to that by $C$, the assumption that $C_0$ is non-degenerate implies that so is $C$. Thus, $C$ is the required $p$-coloring. 
\end{proof}

\begin{lemma}\label{lem920-1}
Let $K$ and $F$ be a classical knot and a surface-knot, respectively. Let $D$ be a diagram. 
Let $p>1$ be an integer. 
If $D$ admits a non-trivial $p$-coloring, then there exists a divisor $q>1$ of $p$ such that $D$ is $q$-colorable. 
\end{lemma}

\begin{proof} 
Assume that $D$ admits a non-trivial $p$-coloring $C$. 
By Lemma \ref{lem920-5}, we assume that there is an arc or sheet of $D$ with the color $0 \pmod{p}$. Then, by Lemma \ref{lem12}, by the dihedral quandle operation $*$, colors of $D$ generate a subgroup $H$ of the cyclic group $\mathbb{Z}/p\mathbb{Z}$. 
Since the $p$-coloring $C$ is non-trivial, $H$ is a non-trivial group, and hence the order of $H$ is greater than $1$ and a divisor of $p$. Take $q$ as the order of $H$, and put $r=p/q$ and take a $q$-coloring $C'$ defined by $C'(\alpha)=C(\alpha)/r \pmod{q}$ for any arc or sheet $\alpha$.  Then, $C'$ gives a non-degenerate $q$-coloring, and we have the required result. 
\end{proof}

\begin{lemma}\label{lem12}
For an integer $n>0$, 
let $d_1, \ldots, d_n$ be non-zero integers.
Then, by the binary operation $*$ defined by $l*m=2m-l$ $(l, m \in \mathbb{Z})$, $0,d_1, \ldots, d_n$ generate $d\mathbb{Z}=\{0, \pm d, \pm 2d, \ldots\} \subset \mathbb{Z}$, where $d=\gcd \{d_1, \ldots, d_n\}$. 
\end{lemma}

\begin{proof}
We denote by $H$ the set generated by $\{0,d_1, \ldots, d_n\}$, using the operation $*$. By the definition of $*$, we see that $H$ is contained in the $\mathbb{Z}$-ideal generated by $\{0,d_1, \ldots, d_n\}$, which is $d \mathbb{Z}$. 
Thus it suffices to show that  $d \mathbb{Z} \subset H$. 

First we show the result for the case $n=1$. 
Since $-l=l*0$ for any integer $l$, taking $l=d_1$, we see that $0, \pm d_1 \in H$. 
For a positive integer $m$, assume that $0, \pm d_1, \pm 2d_1, \ldots, \pm (2m-1)d_1 \in H$. 
Then, since $2md_1=0*(md_1)$ and $(2m+1)d_1=(-d_1)*(md_1)$, and $-l=l*0$ for any $l$, 
by induction on $m$, we see that $d_1 \mathbb{Z} \subset H$. 

Next we show the result for the case $n=2$. 
By the above argument, $d_1 \mathbb{Z} \cup d_2 \mathbb{Z} \subset H$. 
Thus it suffices to show that $d$ is generated by elements in $d_1 \mathbb{Z} \cup d_2 \mathbb{Z}$, using the operation $*$. 
Put $d_1'=d_1/d$, and $d_2'=d_2/d$, which are coprime integers.
Take integers $l, m$ such that $ld_1'+md_2'=1$. 
Since $d_1'$ and $d_2'$ are coprime, one of them is odd. 
First we consider the case when both of $d_1'$ and $d_2'$ are odd. Then one of $l,m$ is even. Assume that $l$ is even. Then, $d=ld_1+md_2=(-md_2)*(l/2)d_1$; hence $d \in H$. 
Next we consider the other case when 
one of $d_1', d_2'$ is odd and the other is even. 
Assume that $d_1'$ is odd and $d_2'$ is even. 
Then, $l$ is odd, and $m$ is odd or even.
If $m$ is even, then $d=(-ld_1)*(m/2)d_2 \in H$.  
If $m$ is odd, then $m-d_1'$ is even. 
Hence $d=(l+d_2')d_1+(m-d_1')d_2=\{-(l+d_2')d_1\}*\{(m-d_1')/2\}d_2$; 
thus $d \in H$. 

The other cases $n>2$ are shown by a similar argument. 
\end{proof}
 
\begin{proposition}\label{prop916-6}
Let $K$ and $F$ be a classical knot and a surface-knot, respectively. Let $p>1$ be an integer. If $K$ or $F$ is $p$-colorable, then the knot determinant $|\Delta_K(-1)|$ or $\Delta_F|_{t=-1}$ is divisible by $p$. 
\end{proposition}
 
\begin{proof}
First we consider the case of a classical knot $K$. 
We assume that $K$ is $p$-colorable. 
Let $D$ be a diagram, and 
let $A(t)$ be the Alexander matrix associated with $\Lambda$-colorings of $D$.  
Let $n$ and $m$ be the numbers of rows and columns of $A(t)$, respectively. 
Remark that since $K$ is $p$-colorable, $K$ is a non-trivial knot, and $m>1$. 
For any row of $A(t)$, by (\ref{eq-0616-02}), its elements form a multiset that is one of the following: 
\begin{eqnarray*}
& & \{t, 1-t, -1, 0,0,\ldots, 0\},\\
& & \{1, -1, 0,0,\ldots, 0\},\\
& & \{t, -t, 0,0,\ldots, 0\},\\
& & \{1-t, -(1-t), 0,0,\ldots, 0\},\\
& & \{0,0,\ldots, 0\}.
\end{eqnarray*}
Hence, the sum of elements of each row is zero, so, by adding the $j$th column once for each $j=1, \ldots, m-1$ to the $m$th column, we transform $A(t)$ to a matrix whose $m$th column is the zero column vector. Then we put $t=-1$, and we consider the resulting integer matrix. By elementary transformations for integer matrices, without using the $m$th column, we transform the resulting matrix to the following Smith normal form $B$ such that  
 
\begin{equation}\label{eq13}
B=P\cdot A(-1)\cdot Q=\begin{pmatrix}
d_1 & 0 & \cdots & 0  & 0\\
0 & d_2 & &\vdots  & \vdots \\
\vdots  &  & \ddots &  &0 \\

0 &  \cdots & 0 & d_{m-1} & 0 \\
0 & 0 & \cdots & 0 & 0\\
\vdots & \vdots & & \vdots & \vdots\\
0 & 0 & \cdots & 0 & 0
\end{pmatrix}, 
\end{equation}
where $P$ and $Q$ are unimodular matrices such that 
\[
Q=\begin{pmatrix}
Q' & *\\
\mathbf{o} & 1
\end{pmatrix}
\]
for some $(n-1, m-1)$-matrix $Q'$, and $d_1, \ldots, d_{m-1}$ are positive integers such that $d_{j+1}$ is a multiple of $d_j$ $(j=1,\ldots, m-2)$, and $d_1 \cdots d_{m-1}=|\Delta_K(-1)|$. 

Suppose that $p$, which is an integer greater than $1$, is not a divisor of $|\Delta_K(-1)|$. 
We consider a system of homogeneous linear equations: 
\begin{equation}\label{eq916-3}
B \mathbf{x}=\mathbf{o} \pmod{p}.  
\end{equation} 
Put $\mathbf{x}=(x_1, \ldots, x_{m})^T$. 
Then, we fix the free variable $x_{m}$ by putting $x_{m}=0$, and we consider solutions of modulo $p$, $\mathbf{x} \in (\mathbb{Z}/p\mathbb{Z})^{m}$, of (\ref{eq916-3}). 
We see that for any solution of (\ref{eq916-3}), each element $x_j$ $(j=1,\ldots, m-1)$ is a multiple of $p/ \gcd \{p, d_j\} \pmod{p}$; hence 
the elements $x_1, \ldots, x_{m-1}, x_{m}$ ($x_{m}=0$) are in the subgroup of the cyclic group $\mathbb{Z}/p \mathbb{Z}$ that is generated by $\{p/ \gcd \{p, d_j\} \pmod{p} \mid j=1,\ldots, m-1\}$, which will be denoted by $H$. Since $d_{j+1}$ is a multiple of $d_j$ $(j=1,\ldots,m-2)$, $p/ \gcd \{p, d_{j+1}\}$ is a divisor of $p/ \gcd \{p, d_j\}$, hence $H$ is generated by $p/ \gcd \{p, d_{m-1}\} \pmod{p}$. Since $p$ is not a divisor of $|\Delta_K(-1)|$, $ 0<\gcd \{p, d_{m-1}\}< p$, and hence the generator $p/ \gcd \{p, d_{m-1}\} \pmod{p} \neq 1 \pmod{p}$ and $H$ is a proper subgroup of $\mathbb{Z}/p \mathbb{Z}$.

Since $B=P\cdot A(-1)\cdot Q$, 
the set of solutions $\mathbf{x}=(x_1, \ldots, x_m)^T\in (\mathbb{Z}/p\mathbb{Z})^m$ of 
the equation $B \mathbf{x}=\mathbf{o} \pmod{p}$ with $x_m=0$ has one-to-one correspondence to the set of solutions $ \mathbf{x}'=(x_1', \ldots, x_m')^T\in (\mathbb{Z}/p\mathbb{Z})^m$ of 
the equation $A(-1) \cdot \mathbf{x}'=\mathbf{o} \pmod{p}$ with $x'_m=0$. 
We consider elements of any solution of this equation $A(-1) \cdot \mathbf{x}'=\mathbf{o} \pmod{p}$, with $x'_m=0$. 
Since the group $H$ is closed under addition and multiplication by integers, 
$x'_1, \ldots, x'_{m}$ of any solution $\mathbf{x}'$ are also in $H$. 
Further, by Lemma \ref{lem12},  $H$ is closed under the dihedral quandle operation $*$. 
 Hence, $x_1', \ldots, x_m'$ do not generate $R_p$ by the operation $*$. By Lemma \ref{lem-0613-5}, this implies the following. Let $\alpha$ be the $m$-th arc. Then, for any $p$-coloring $C$ satisfying $C(\alpha)=0$, colors of arcs given by $C$ do not generate $R_p$ by the operation $*$: $D$ admits no non-degenerate $p$-colorings $C$ satisfying $C(\alpha)=0$. 
This contradicts Lemma \ref{lem920-5}, 
hence $|\Delta_K(-1)|$ is divisible by $p$. 

Next we consider the case of a surface-knot $F$.  We assume that $F$ is $p$-colorable. We consider $A(-1)$ and transform it to the Smith formal form.  
Since the knot determinant $\Delta(F)|_{t=-1}$ of any surface-knot is odd \cite[Proposition 5.8]{Joseph}, we see that 
the Smith normal form is the same form with (\ref{eq13}), 
where we replace the knot determinant $|\Delta_K(-1)|$ by $\Delta_F|_{t=-1}$. 
Hence, by the same argument as the case of $K$, we see that $\Delta(F)|_{t=-1}$ is divisible by $p$. 
\end{proof}

\begin{remark}
We can define the Goeritz invariant \cite{Goeritz} for surface-knots and give a necessary and sufficient condition for surface-knots to be $p$-colorable. See also \cite{P}. 
\end{remark}

For a surface-knot $F$, we denote the knot determinant $\Delta(F)|_{t=-1}$ by $\Delta(F)|_{-1}$. 
\begin{proposition}\label{prop0614}
Let $F=\mathcal{S}_n(a,b)$ be a torus-covering $T^2$-knot with basis $n$-braids $(a, b)$ such that the closure of the braid $a$ is a knot. 
If $\mathcal{S}_n(a,b)$ is $|\Delta_{\hat{a}}(-1)|$-colorable, then 
$\Delta(F)|_{-1}=|\Delta_{\hat{a}}(-1)|$. 
\end{proposition}

\begin{proof}
Assume that $\mathcal{S}_n(a,b)$ is $|\Delta_{\hat{a}}(-1)|$-colorable. 
Then, by Proposition \ref{prop916-6}, $\Delta(F)|_{-1}$ is divisible by $|\Delta_{\hat{a}}(-1)|$; hence, in particular, $|\Delta_{\hat{a}}(-1)| \leq \Delta(F)|_{-1}$. 
Now, by Theorems \ref{thm-Lin} and \ref{thm2-1}, $(|\Delta_{\hat{a}}(-1)|-1)/2$ and  $((\Delta(F)|_{-1})-1)/2$ are  the numbers of conjugacy classes of irreducible metabelian $SU(2)$-representations of $G=\pi_1(S^3 \backslash \hat{a})$ and $G'=\pi_1(S^4 \backslash F)$, respectively. By the presentations given in Lemmas \ref{lem-Lin} and \ref{lem-k}, we see that $G'$ is a quotient group of $G$, and hence the number of representations of $G'$ is smaller than or equal to that of $G$; hence $\Delta(F)|_{-1} \leq |\Delta_{\hat{a}}(-1)|$. 
Thus $\Delta_F|_{-1}=|\Delta_{\hat{a}}(-1)|$. 
\end{proof}
 
Using an argument similar to that in the proof of Proposition \ref{prop916-6}, we have the following proposition.  
\begin{proposition}\label{prop916-5}
Let $K$ and $F$ be a classical knot and a surface-knot, respectively. 
Let $p$ be an odd prime. 
If $K$ or $F$ is $p$-colorable and not $q$-colorable for any integer $q \neq p$ $(q>1)$, then $|\Delta_K(-1)|=|\mathrm{Col}_p(K)|/p$ or $\Delta(F)|_{-1}=|\mathrm{Col}_p(F)|/p$.  
\end{proposition}

\begin{proof}
We show the result for the case of $K$. The case of $F$ is the same. Let $D$ be a diagram. Assume that $D$ is $p$-colorable and not $q$-colorable for any integer $q \neq p$ ($q>1$). 
We consider the system of linear equations 
\begin{equation}\label{eq923-1}
A(-1) \cdot \mathbf{x}=\mathbf{o}, 
\end{equation}
where $A(t)$ is the Alexander matrix associated with $\Lambda$-colorings of $D$. 

Let $r>1$ be an integer.  
Since an $r$-coloring corresponds to a solution modulo $r$ of (\ref{eq923-1}), we count the number of solutions modulo $r$. 
We fix an arc $\alpha$ of $D$. 
In order to arrange that $r$-colorings contain exactly one trivial coloring, we consider $r$-colorings $C$ which satisfies the condition $C(\alpha)=0$, which will be called $r$-colorings {\it satisfying condition O}. 
An $r$-coloring satisfying condition O corresponds to a solution modulo $r$ of (\ref{eq923-1}) such that the last element $x_m$ of $\mathbf{x}=(x_1,\ldots, x_m)^T$ is determined by $0$, where $m$ is the number of columns of $A(-1)$, and we arrange arcs and take    $\alpha$ as the $m$th arc; 
we call the equation after fixing $x_m=0$ the equation {\it satisfying condition O}. 
We remark that the set of $r$-colorings satisfying condition O contain exactly one trivial coloring, whose colors are all $0 \in \mathbb{Z}/p\mathbb{Z}$, so the trivial coloring corresponds to the trivial solution of (\ref{eq923-1}) satisfying condition O. 
By Lemma \ref{lem10}, we see that for any integer $r>1$, the number of solutions mudulo $r$ of the equation (\ref{eq923-1}) satisfying condition O is $1$ or $|\mathrm{Col}_p(K)|/p$. 

Now, we calculate (\ref{eq923-1}), after fixing $x_m=0$. 
By the argument in the proof of Proposition \ref{prop916-6},  we see that the number of solutions modulo $r$ of (\ref{eq923-1}) satisfying condition O is the number of solutions modulo $r$ of 
\begin{equation}\label{eq-529}
\begin{pmatrix}
d_1 & 0 & \cdots & 0 \\
0 & d_2 & &\vdots \\
\vdots  &  & \ddots & \\

0 &  \cdots & 0 & d_{{m-1}}
\end{pmatrix} 
\mathbf{x}=\mathbf{o}, 
\end{equation}
where $d_1, \ldots, d_{m-1}$ are positive integers such that $d_1\cdots d_{m-1}=|\Delta_K(-1)|$. 
We denote the number of solutions in question by $N$. 
Take $r=|\Delta_K(-1)|$. Then, solutions of  (\ref{eq-529}) modulo $|\Delta_K(-1)|$ 
are given by 
\begin{equation*}
\begin{split}
\mathbf{x}=(n_1 d_2 d_3\cdots d_{m-1}, n_2 d_1 d_3 \cdots d_{m-1}, \ldots, n_{m-1} d_1d_2\cdots d_{m-2})^T \\
\pmod{|\Delta_K(-1)|},
\end{split} 
\end{equation*}
where $n_j=0,1,\ldots, d_j-1$ $(j=1, \ldots, m-1)$, and $\mathbf{x} \in (\mathbb{Z}/r \mathbb{Z})^{m-1}$ with $r=|\Delta_K(-1)|$. 
Hence we see that $N=d_1d_2 \cdots d_{m-1}$, which equals $|\Delta_K(-1)|$; thus, after fixing $x_m=0$, the equation (\ref{eq923-1}) has $|\Delta_K(-1)|$ solutions modulo $|\Delta_K(-1)|$. Since $K$ is $p$-colorable, by Proposition \ref{prop916-6}, $|\Delta_K(-1)|>1$. 
Therefore we see that $|\Delta_K(-1)|=|\mathrm{Col}_p(K)|/p$. 
\end{proof}

\begin{lemma}\label{lem10}
Let $K$ and $F$ be a classical knot and a surface-knot, respectively. 
We fix a diagram $D$ and an arc or sheet $\alpha$. 
We say that an $r$-coloring $C$ satisfies condition O, if $C$ satisfies $C(\alpha)=0$. 
For an integer $r>1$, we denote by $|\mathrm{Col}_r^0|$ the number of $r$-colorings satisfying condition O. 

Let $p$ be an odd prime. 
Assume that $K$ or $F$ is $p$-colorable. 
Then, we have the following. 

\begin{enumerate}[$(1)$]
\item 
We have $|\mathrm{Col}_p^0|=|\mathrm{Col}_p(K)|/p$ or $|\mathrm{Col}_p(F)|/p$, which is $p^k$ for some positive integer $k$. 

\item
The knot 
$K$ or $F$ is not $q$-colorable for every integer $q \neq p$ $(q>1)$ if and only if the set consisting of $|\mathrm{Col}_r^0|$ $(r>1)$ is 
$\{1, |\mathrm{Col}_p^0|\}$. 
\end{enumerate}
\end{lemma}

We remark that Proposition \ref{prop916-5} is the combination of Lemma \ref{lem10} (1) and the following proposition. 

\begin{proposition}\label{prop-0611}
We use the notation in Lemma \ref{lem10}. 
Let $K$ and $F$ be a classical knot and a surface-knot, respectively. 
Let $p$ be an odd prime. 
If $K$ or $F$ is $p$-colorable and not $q$-colorable for any integer $q \neq p$ $(q>1)$, 
then $|\Delta_K(-1)|$ or $\Delta(F)|_{-1}$ equals $|\mathrm{Col}_p^0|$. 
\end{proposition}

\begin{proof}[Proof of Lemma \ref{lem10}]
We show the result for the case of $K$. The case of $F$ is the same. Recall that $D$ is a diagram.  

We show (1). 
It is known by Iwakiri \cite{Iwakiri} that the set of $p$-colorings $\mathrm{Col}_p(D)$ is a linear space over $\mathbb{Z}/p\mathbb{Z}$. 
Further, since $D$ admits a non-trivial $p$-coloring, $|\mathrm{Col}_p(K)|>p$. Hence $|\mathrm{Col}_p(K)|=p^l$ for some integer $l$ with $l>1$. 
After fixing the color $C(\alpha)$, we have $|\mathrm{Col}_p(K)|/p$ $p$-colorings satisfying condition O. 
Hence, $|\mathrm{Col}_p^0|=|\mathrm{Col}_p(K)|/p=p^k$ for some positive integer $k$ $(k=l-1)$. 

Next we show (2). 
We show the only if part. 
Assume that $K$ is $p$-colorable and not $q$-colorable for every $q \neq p$ ($q>1$). 
Let $r$ be an integer $r>1$. 
We consider the case when $r \not\equiv 0 \pmod{p}$. 
If $D$ admits a non-trivial $r$-coloring, then, 
by Lemma \ref{lem920-1}, $D$ is $d$-colorable for some divisor $d>1$ of $r$. 
However, by assumption, 
$D$ is not $q$-colorable for any $q \neq p$ $(q>1)$, 
so $D$ is not $d$-colorable for any divisor $d>1$ of $r$ with $r \not\equiv 0 \pmod{p}$. 
Hence $D$ admits only trivial $r$-colorings for $r \not\equiv 0 \pmod{p}$; thus the number of $r$-colorings satisfying condition O is $1$: $|\mathrm{Col}_{r}^0|=1$ $(r \not\equiv 0 \pmod{p})$.  

We consider the case $r=sp$ for an integer $s>1$. 
By the argument in the proof of Lemma \ref{lem920-1}, 
if $D$ admits a non-trivial $sp$-coloring $C$ satisfying condition O, then colors of arcs generate a non-trivial subgroup of $\mathbb{Z}/sp\mathbb{Z}$, which induces a non-degenerate $q$-coloring $C'$ of $D$, satisfying condition O, where $q$ is the order of the subgroup, which is a divisor of $sp$ with $1<q \leq sp$. Further,  $C=(sp/q)C'$, which assigns to each arc $\alpha$ the color $(sp/q) \cdot C'(\alpha)$ $\pmod{sp}$. 
Since $D$ is $p$-colorable but not $d$-colorable for any divisor $d$ of $sp$ with $d \neq 1,p$, we see that $q=p$, and 
each non-trivial $sp$-coloring $C$ satisfying condition O 
is written as $C=sC'$, where $C'$ is 
a non-degenerate $p$-coloring satisfying condition O. 
Since $p$ is prime, by Lemma \ref{lem920-1} again, any non-trivial $p$-coloring satisfying condition O is a non-degenerate $p$-coloring satisfying condition O. Hence, the set of $sp$-colorings satisfying condition O and the set of $p$-colorings satisfying condition O have one-to-one correspondence, which implies $|\mathrm{Col}_{sp}^0|=|\mathrm{Col}_p^0|$ $(s\geq 1)$. 
Thus we have the only if part. 

Next we show the if part. 
Assume that $K$ is $p$-colorable, and, for any integer $r>1$, the number of $r$-colorings satisfying condition O is either $1$ or $|\mathrm{Col}_p^0|$.  

Let $r>1$ be an integer. 
Suppose that $K$ is $r$-colorable, that is, $D$ admits a non-degenerate $r$-coloring $C$.  Take any divisor $d$ of $r$ with $d>1$.  
Then, the non-degenerate $r$-coloring $C$ induces a non-degenerate $d$-coloring $C'$, given by $C'(\alpha)=C(\alpha) \pmod{d}$ for any arc $\alpha$. Thus, $K$ is $d$-colorable. We choose $d$ to be a prime. 
If $d$ is odd prime, then, by (1) of this lemma, the number $|\mathrm{Col}_d^0|$ of $d$-colorings satisfying condition O is $d^l$ for a positive integer $l$.  
By assumption, $|\mathrm{Col}_d^0|$ is either $1$ or $|\mathrm{Col}_p^0|$. Since $d^l>1$, we see that $d^l=|\mathrm{Col}_p^0|$. Since $|\mathrm{Col}_p^0|=p^k$ for some positive integer $k$ by (1) of this lemma, we see that the primes $d$ and $p$ coincide. 
If $d=2$, then, for a non-degenerate $2$-coloring $C$, there appear both colors $0$ and $1$, and all over-arcs (respectively, under-arcs) are colored by the same color. Thus $K$ has at least 2 components, which is a contradiction. Hence, the case $d = 2$ does not occur.  Thus, if $K$ is $r$-colorable, then $r$ is divisible by $p$; hence $K$ is not $r$-colorable for $r \not\equiv 0 \pmod{p}$. 

We consider $sp$-colorings of $D$, satisfying condition O, for an integer $s>1$. 
Let $H$ be the set of $sp$-colorings induced from the set of $p$-colorings satisfying condition O, such that each element of $H$ is given by $sC$ for each $p$-coloring $C$ satisfying condition O. 
Then, $H$ and 
the set of $p$-colorings satisfying condition O have one-to-one correspondence. Further, each $sp$-coloring in $H$ is degenerate.
If $K$ is $sp$-colorable, then $D$ admits a non-degenerate $sp$-coloring, which implies that $H$ is a proper subset of the set of $sp$-colorings satisfying condition O; hence $|\mathrm{Col}_{sp}^0|>|\mathrm{Col}_p^0|$, which contradicts the assumption.  
Thus, the assumption implies that $K$ is $p$-colorable but not $sp$-colorable $(s>1)$. Hence we see that the if part holds true. 
\end{proof}

\begin{proof}[Proof of Theorem \ref{thm2-2}]
The required result follows from Theorem \ref{thm2-1} and Proposition \ref{prop0614}. 
\end{proof}

\begin{proof}[Proof of Theorem \ref{thm2-3}]
The required result follows from Theorem \ref{thm2-1} and Proposition \ref{prop916-5}.
\end{proof}

\begin{proof}[Proof of Corollary \ref{cor2-4}]
We use the notation in Lemma \ref{lem10}. 
We use the same notation $c$, $\hat{c}$, or $F=\mathcal{S}_n(c,\tau^{lm})$ for its diagram, and for an integer $r>1$, we denote $|\mathrm{Col}_r^0|$ for $\hat{c}$ and $F$ by $|\mathrm{Col}_r^0(\hat{c})|$ and   
$|\mathrm{Col}_r^0(F)|$, respectively.
Remark that the knot $\hat{c}$ is the connected sum of $n-1$ copies of a $(2,p)$-torus knot. We can see the following property for $\hat{c}$ by seeing the Alexander matrix of a connected sum of $(2,p)$-torus knots, but here we will use the following method. 
We investigate what $r$-colorings satisfying condition O exist for $\hat{c}$ for each integer $r>1$. 
By fixing the color $0$ to the $n$th initial arc of $c$, and giving colors $x_1, \ldots, x_{n-1}$ to the other initial arcs of $c$, and seeing the condition so that colors of the terminal arcs agree with those of the initial arcs, we see the following.  
The knot $\hat{c}$ is $p$-colorable, and $p$-colorings are determined by colors of the initial arcs of $c$, which form a basis of the linear space $\mathrm{Col}_p(\hat{c})$ over $\mathbb{Z}/p\mathbb{Z}$, where $\mathrm{Col}_p(\hat{c})$ is the set consisting of $p$-colorings of $\hat{c}$. 
Hence, we see that $|\mathrm{Col}_p(\hat{c})|=p^n$ and 
$|\mathrm{Col}_p^0(\hat{c})|=p^{n-1}$. 
Further, for any integer $r>1$ with $r \not\equiv 0 \pmod{p}$, we see that the set of $r$-colorings satisfying condition O is the set consisting of exactly one element, which is the trivial $r$-coloring satisfying condition $O$. And for any integer $s>1$, we see that 
each $p$-coloring $C$ satisfying condition O induces an $sp$-coloring $sC$ satisfying condition O, and there are no other $sp$-colorings satisfying condition O. 

We investigate $r$-colorings of $F$, satisfying condition O, for each integer $r>1$. 
For $r \not\equiv 0 \pmod{p}$, since the set of $r$-colorings satisfying condition O for $\hat{c}$ is the set consisting only of the trivial $r$-coloring satisfying condition $O$, so is for $F$. Thus, $|\mathrm{Col}_r^0(F)|=1$ for any $r \not\equiv 0 \pmod{p}$. 

By \cite[Sections 5 and 6]{N5}, $p$-colorings of $F$ are induced from those of $\hat{c}$, and are determined from colors of the initial arcs of $c$. 
For $r=sp$ $(s\geq 1)$, since each $sp$-coloring $C'$ satisfying condition O for $\hat{c}$ is induced from a $p$-coloring $C$ satisfying condition O such that $C'=sC$, so is for $F$. 
Thus, $F$ is $p$-colorable and $|\mathrm{Col}_{p}(F)|=|\mathrm{Col}_p(\hat{c})|=p^n$, and $|\mathrm{Col}_{r}^0(F)|=|\mathrm{Col}_p^0(\hat{c})|(=p^{n-1})$ for any $r \equiv 0 \pmod{p}$. 

Hence, $F$ is $p$-colorable and the set consisting of $|\mathrm{Col}_r^0(F)|$ $(r>1)$ is $\{1, |\mathrm{Col}_p^0(F)|\}$, and by Lemma \ref{lem10} (2), $F$ is not $q$-colorable for every integer $q \neq p$ $(q>1)$. 
Thus, by Theorem \ref{thm2-3}, 
since $|\mathrm{Col}_{p}(F)|=p^n$, the number in question is $(p^{n-1}-1)/2$. 
\end{proof}

\section*{Acknowledgements}
The author was partially supported by JSPS KAKENHI Grant Number 19K03464. The author would like to thank the referee for kindly reviewing the manuscript.

\end{document}